\documentclass[11pt,reqno,a4paper]{amsart}
\usepackage{amsfonts,amsmath,amsthm}
\usepackage{tikz}
\usepackage{times}

\definecolor{blau}{rgb}{0,0,0.75} 
\usepackage[pdftex]{hyperref}
\hypersetup{colorlinks,linkcolor=blau,citecolor=blue,urlcolor=blau}

\allowdisplaybreaks

\newtheorem{theorem}{Theorem}
\newtheorem{lemma}[theorem]{Lemma}
\newtheorem{coroll}{Corollary}
\newtheorem{prop}{Proposition}

\theoremstyle{definition}
\newtheorem{remark}{Remark}
\newtheorem{example}{Example}

\newcommand{\fallfak}[2]{\ensuremath{#1^{\underline{#2}}}}

\newcommand{\N}{\ensuremath{\mathbb{N}}}

\newcommand{\C}{\ensuremath{\mathbb{C}}}
\newcommand{\R}{\ensuremath{\mathbb{R}}}
\newcommand{\Z}{\ensuremath{\mathbb{Z}}}
\newcommand{\ith}[1]{\ensuremath{ {#1}\textsuperscript{th}}}

\DeclareMathOperator{\lemsine}{sl}

\DeclareMathOperator{\odeg}{deg}
\DeclareMathOperator{\erf}{erf}
\begin{document}

\author[M.~Kuba]{Markus Kuba}
\address{Markus Kuba\\
Institute of Applied Mathematics and Natural Sciences\\
University of Applied Sciences - Technikum Wien\\
H\"ochst\"adtplatz 5, 1200 Wien} %
\email{kuba@technikum-wien.at}

\author{Alois Panholzer}
\address{Alois Panholzer\\
Institut f{\"u}r Diskrete Mathematik und Geometrie\\
Technische Universit\"at Wien\\
Wiedner Hauptstr. 8-10/104\\
1040 Wien, Austria} \email{Alois.Panholzer@tuwien.ac.at}

\thanks{The second author was supported by the Austrian Science Foundation FWF, grant P25337-N23.}

\title[Combinatorial families of multilabelled increasing trees]{Combinatorial families of multilabelled increasing trees and hook-length formulas}

\keywords{Hook-length formulas, Bilabelled trees, Multilabelled trees, Increasing trees, Elliptic functions, Weierstrass-$\wp$ function, Differential equations}%
\subjclass[2000]{05C05, 05A15, 05A19} %

\begin{abstract}
In this work we introduce and study various generalizations of the notion of increasingly labelled trees, where the label of a child node is always larger than the label of its parent node, to multilabelled tree families, where the nodes in the tree can get multiple labels.

For all tree classes we show characterizations of suitable generating functions for the tree enumeration sequence via differential equations. Furthermore, for several combinatorial classes of multilabelled increasing tree families we present explicit enumeration results.
We also present multilabelled increasing tree families of an elliptic nature, where the exponential generating function can be expressed in terms of the Weierstrass-$\wp$ function or the lemniscate sine function.

Furthermore, we show how to translate enumeration formulas for multilabelled increasing trees into hook-length formul\ae{} for trees and present a general ``reverse engineering'' method to discover hook-length formul\ae{} associated to such tree families.
\end{abstract}

\maketitle

\begin{center}
  \large{\emph{Dedicated to Helmut Prodinger on the occasion of his 60th birthday}}
\end{center}

\section{Introduction\label{HookBijII-sec1}}

\subsection{Multilabelled increasing tree families}
Increasing trees or increasingly labelled trees are rooted labelled trees, where the nodes of a tree $T$ of size $|T| = n$ (where the size $|T|$ of a tree denotes the number of vertices of $T$) are labelled with distinct integers from a label set $\mathcal{M}$ of size $|\mathcal{M}|=n$ (usually, one chooses as label set the first $n$ positive integers, i.e., $\mathcal{M} = [n] := \{1, 2, \dots, n\}$) in such a way that the label of any node in the tree is smaller than the labels of its children. As a consequence, the labels of each path from the root to an arbitrary node in the tree are forming an increasing sequence, which explains the name of such a labelling.

Various increasing tree models turned out to be appropriate in order to describe the growth behaviour of quantities in various applications and occurred in the probabilistic literature, see \cite{MahSmy1995} for a survey collecting results prior 1995. E.g., they are used to describe the spread of epidemics, to model pyramid schemes, and as a simplified growth model of the world wide web.

First occurrences of increasing trees in the combinatorial literature were due to bijections to other fundamental combinatorial structures, e.g., binary increasing trees and more generally $(d+1)$-ary increasing trees of size $n$ are in bijection to permutations and so-called $d$-Stirling permutations of order $n$, respectively, see \cite{Fra1976,JanKubPan2011,Par1994,Sta1986} and references therein. A further example are increasingly labelled non-plane unary-binary trees, where it turned out that the number of such trees of size $n$ is twice the number of alternating permutations of order $n$, see \cite{Don1975,KuzPakPos1994}.

A first systematic treatment of increasing labellings (and related, so-called monotone labellings, where labels are not necessarily distinct) of trees is given by Pro\-din\-ger and Urbanek in \cite{ProUrb1983}, where in particular plane increasing trees (i.e., increasingly labelled ordered trees) of a given size could be enumerated. A fundamental study of increasing tree families yielding exact and asymptotic enumeration results as well as a distributional analysis of various tree parameters is given in \cite{BerFlaSal1992}, see also \cite{FlaSed2009} and references therein.

In above definition of increasing trees each node in the tree gets exactly one label. In this work we introduce and study several extensions of this concept to multilabelled tree families, i.e., where the nodes in the tree are equipped with a set or a sequence of labels. Whereas (unilabelled) increasing trees are studied extensively in the combinatorial and probabilistic literature, best to our knowledge the enumeration of increasingly multilabelled trees has not been addressed so far (apart from the author's work \cite{KubPan2012}, where a particular instance appears).

In bilabelled increasing trees or increasingly bilabelled trees each node in the tree gets a set of two labels and the labels of a child node are always larger than both of the labels of its parent node. Such a labelling has been introduced in \cite{KubPan2012} in order to provide a combinatorial explanation of a certain hook-length formula (a summation formula for the trees of a given size in the tree family considered, where the hook-lengths, i.e., the number of descendants, of the nodes in each tree are occurring) for unordered labelled trees. Here we will give a systematic treatment of families of bilabelled increasing trees, which relies on a general symbolic combinatorial description of such tree classes leading to an implicit characterization of the exponential generating function of the number of bilabelled increasing trees of size $n$ (whose nodes are bilabelled with distinct integers of the label set $[2n]$). As a consequence we will present new enumerative results, extending the known result~\cite{KubPan2012} concerning unordered bilabelled increasing trees, where our focus is here on tree classes yielding interesting enumeration formul\ae. E.g., for ordered bilabelled increasing trees we are able to express the generating function in terms of the antiderivative of the inverse error function, whereas for increasingly bilabelled $3$-bundled trees (increasing bilabellings of trees from a certain family of so-called generalized plane-oriented trees, see~\cite{KubPan2007}) we even get a simple closed-form enumeration result.

We also present several families of bilabelled increasing trees of an \emph{elliptic nature}. Combinatorial families whose generating functions are of an elliptic nature have occasionally appeared in the literature: Dumont~\cite{Dumont1979,Dumont1981}, Flajolet~\cite{Flajo1981}, Fran\c{c}on~\cite{FlaFran1989} and Viennot~\cite{Viennot1980} studied some models related to permutations. 
Panholzer and Prodinger~\cite{PanPro1998} uncovered the elliptic nature of fringe-balanced binary search trees. Flajolet et al.~\cite{FlaGabPek2005} discussed urn models whose history generating functions can be expressed in terms of the Weierstrass-$\wp$ function.
The Weierstrass-$\wp$ function also appeared recently in the work of Bouttier et al.~\cite{Boutt2003II}, and Drmota~\cite{Drmota2013}: they studied the support of the random measure ISE and its relation to embedded tree families. In this article we discuss families of increasingly bilabelled trees, whose generating functions can be expressed in terms of the Weierstrass-$\wp$ function or related elliptic functions such as the lemniscate-sine function. In particular, strict-binary bilabelled increasing trees and unordered bilabelled increasing trees with only even degrees are of an elliptic type. Furthermore, we show that arbitrary families of binary or ternary bilabelled increasing trees have an elliptic nature.

It is well known (see, e.g., \cite{KubPan2013}) that enumerative results for (unilabelled) increasing trees can be transferred into hook-length formul\ae{} for the corresponding tree families. Based on the fact that the number of increasing bilabellings of a given tree can be described by a nice product formula containing the hook-lengths of the nodes in the tree, we will show that the enumeration of bilabelled increasing trees gives rise to hook-length formul\ae{} in a natural way. In this context we also present a ``reverse engineering approach'', i.e., a general method for discovering hook-length formul\ae{} associated to families of bilabelled increasing trees. 

The concept of increasing bilabellings of trees can be extended in a natural way to $k$-labelled trees: in a $k$-labelled increasing tree (or increasingly $k$-labelled tree) of size $n$ a total of $k n$ distinct labels is distributed amongst the vertices of the tree such that each node gets a set of exactly $k$ labels and each label of a child node is larger than all labels of its parent node. Such labellings also lead to hook-length formul\ae{} and the before-mentioned reverse engineering approach could be extended to $k$-labelled increasing trees. 
We discuss trilabelled unordered increasing trees and relate the generating function 
with the solution of the so-called Blasius differential equation 
\[
y'''(z)+y''(z)y(z)=0,\qquad y(0)=0,\quad y'(0)=0,\quad \lim_{z\to\infty}y'(z)=1,
\]
arising in the study of the Prandtl-Blasius Flow~\cite{Blasius1908,BlasiusBoyd1999,BlasiusFinch2008,BlasiusHager2003}, providing a combinatorial interpretation of the coefficients of the Blasius function $y(z)$.
 
\smallskip 

The labellings introduced so far could be considered as ``label regular'', since each node in the tree gets the same number of labels. However, one can also consider non-regular labellings, where the number of labels a node can get is not fixed. Here it is natural to consider the number $m$ of distinct labels, which are distributed amongst the nodes in a tree, instead of the tree-size $n$. A free multilabelled increasing tree (or increasingly free multilabelled tree) is a tree of size $n$, where the $m = |\mathcal{M}| \ge n$ labels of the label set $\mathcal{M}$ are distributed amongst the nodes in the tree in such a way that each node in the tree gets a non-empty set of labels and each label of a child node is larger than all nodes of its parent node.
We are able to provide a recursive description of the number of free multilabelled increasing trees with a label set of size $m$ for general tree families and we provide explicit enumeration results for a few interesting instances. Although such non-regular labellings are no more directly amenable to hook-length formul\ae, we also give an extension of the notion of the hook-length of a node yielding corresponding results.

One can also consider such multilabelled increasing trees with restrictions on the number of labels a node can get, e.g., we can assume that each node in the tree can only hold up to two labels. Such unilabelled-bilabelled increasing tree families also yield recurrences for the number of trees with a label set of size $m$ and as an example we give an enumeration formula for unordered unilabelled-bilabelled increasing trees.

Finally, we shortly consider another concept of increasing multilabellings of trees, namely so-called $k$-tuple labelled increasing trees, where each node in a tree $T$ of size $n$ gets a $k$-tuple of labels, such that the $j$-th component of the labels of the nodes, with $1 \le j \le k$, are forming an increasing labelling of $T$ with label set $[n]$. The particular instance $k=2$ has been introduced in \cite{KubPan2012} to give a combinatorial explanation of a certain hook-length formula for unordered labelled trees, but also instances $k > 2$ naturally lead to hook-length formul\ae.

\subsection{Weighted ordered tree families\label{ssec:WeightedOrdTrees}}
In order to formulate combinatorial descriptions of the various multilabelled increasing tree families as well as to give hook-length formul\ae{} for them it is advantageous to state the fundamental results for a class of weighted trees, which is known in the literature as simple families of trees, see \cite{FlaSed2009}. The basic objects considered in this context are \emph{ordered trees}, also called 
\emph{planted plane trees}, i.e., rooted trees, where to each node $v$ there is attached a (possibly empty) sequence of child-nodes (thus the left-to-right order of the children is important). 
Throughout this paper we denote by $\mathcal{O}$ the family of ordered trees.
Different simply generated tree models are then obtained by considering weighted ordered trees, where each node $v$ in an ordered tree $T \in \mathcal{O}$ gets a certain weight factor depending on the out-degree of $v$, i.e., the number of children of $v$, and the weight of the tree $T$ is defined as the product of the weight factors of all of its nodes.

More precisely, we define a family $\mathcal{T}$ of \emph{weighted ordered trees} as follows. A sequence of non-negative numbers $(\varphi_{j})_{j \ge 0}$, with $\varphi_{0} > 0$, is used to define the weight $w(T)$ of any ordered tree $T \in \mathcal{O}$ by 
\begin{equation*}
  w(T) := \prod_{v \in T} \varphi_{\odeg(v)},
\end{equation*}
where $v$ ranges over all vertices of $T$ and $\odeg(v)$ is the out-degree of $v$.
The family $\mathcal{T}$ consists then of all ordered trees $T$ (or equivalently of all ordered trees $T$ with $w(T) \neq 0$) together with their weights $w(T)$, i.e., one might think of pairs $(T, w(T))$. 
Furthermore, $\mathcal{T}(n)$ denotes the family of weighted ordered trees of size $n$, i.e., all pairs $(T, w(T))$, with $|T| = n$;
more generally, for a family $\mathcal{C}$ of combinatorial objects, with $\mathcal{C}(n)$ we always denote the set of objects of $\mathcal{C}$ of size $n$.

It is well-known that many important combinatorial tree families such as ordered trees ($\varphi_{j} = 1, j \ge 0$), $d$-ary trees ($\varphi_{j} = \binom{d}{j}$), strict-binary trees ($\varphi_{0} = \varphi_{2} = 1$, and $\varphi_{j}=0$, otherwise) and $d$-bundled trees ($\varphi_{j} = \binom{j+d-1}{j}, j \ge 0$) are equivalent to such weighted tree models, where the degree-weights are chosen in an appropriate way, see, e.g., \cite{FlaSed2009,KubPan2013}.
Of course, one can also consider corresponding (uni)labelled tree families $\widetilde{\mathcal{T}}$, where the nodes of a tree of size $n$ are labelled with distinct integers of $[n]$; such trees can be considered as weighted instances of the family $\widetilde{\mathcal{O}}$ of (uni)labelled ordered trees. In particular, by choosing the degree-weights $\varphi_{j} = \frac{1}{j!}, j \ge 0$, this labelled weighted ordered tree model is equivalent to the family $\widetilde{\mathcal{U}}$ of unordered labelled trees. This justifies in this context the choice of this general tree model: all tree families considered later on can be considered as specifically multilabelled weighted ordered trees, i.e., trees from a weighted ordered tree family $\mathcal{T}$ (or the labelled counterpart $\widetilde{\mathcal{T}}$) are equipped with certain increasing multilabellings.

\subsection{Hook-length formul\ae}
Besides the enumerative interest in families of increasingly multilabelled trees, we present hook-length formul\ae{} associated to the various tree families. Given a rooted tree $T$, we call a node $u \in T$ a \emph{descendant} of node $v \in T$ if $v$ is lying on the unique path from the root of $T$ to $u$. The \emph{hook-length} $h_{v} := h(v)$ of a node $v \in T$ is defined as the number of descendants of $v$ including the node $v$ itself (i.e., it is the size of the subtree rooted at $v$). 

Various hook-length formul\ae{} for different tree families have been obtained recently, see, e.g., \cite{Chen2009,GesSeo2006,Han2008,Pos2009}.
In particular, Han~\cite{Han2008} developed a very versatile expansion technique for deriving
hook-length formul\ae{} for partitions and trees. Han's method for trees was extended by Chen et al.~\cite{Chen2009} and by the authors~\cite{KubPan2013}, 
which allows to determine the ``hook-weight function'' $\rho(n)$ itself from the considered generating function (or, when considering labelled tree families, the corresponding exponential generating function) 
\[
G(z) = \sum_{n \ge 1} \Big(\sum_{T \in \mathcal{T}(n)} \prod_{v \in T} \rho(h_{v})\Big) z^{n},
\]
with $\mathcal{T}(n)$ the set of trees of size $n$ of a family $\mathcal{T}$.
As a prominent example, Han's expansions technique can be used to give a simple proof of the hook-length formula 
\[
\sum_{T \in \mathcal{B}(n)} \prod_{v \in T} \left(1+\frac{1}{h_{v}}\right)= \frac{2^{n} (n+1)^{n-1}}{n!}
\] for the family of \emph{binary trees} $\mathcal{B}$ obtained by Postnikov~\cite{Pos2009}. 

Besides the search and derivation of hook-length formul\ae{} for trees a second important research aspect is to give combinatorial interpretations of them and thus to obtain a ``concrete meaning''. In this work we provide such interpretations in terms of families of increasing multilabelled trees, thus giving concrete realizations for certain hook-length formul\ae{} for labelled ordered trees. In particular, we also obtain ``elliptic hook-length formul\ae{}''. For example, we show that the family $\mathcal{S}$ of strict-binary labelled trees satisfies
\begin{align*}
\sum_{T\in\mathcal{S}(n)}  \frac{1}{\prod_{v \in T}\left(2h_{v}(2h_{v}-1)\right)}
&= 
\frac{n!(2n+1)2^{3n+4}\pi^{n+1}}{3^{\frac{n-1}2}\Gamma^{4n+4}(\frac14)}\\
&\quad \times \sum_{n_1,n_2\in\Z}\frac{1}{(1+n_1+n_2+i(n_1-n_2))^{2n+2}}.
\end{align*}

\subsection{Notation}
The double factorial $(2n-1)!!$ is defined as the product $(2n-1)!! = \prod_{i=1}^n (2i-1)$. 
We denote with $\fallfak{x}{s}=x(x-1)\dots(x-(s-1))$, $s\ge 0$, the falling factorials.
For the reader's convenience we give throughout this work, whenever possible, links to the On-Line En\-cyclopedia of Integer Sequences - \href{https://oeis.org}{OEIS}.

\section{Bilabelled increasing trees\label{sec:BilabelledIncTrees}}

\subsection{Combinatorial description of bilabelled increasing trees}
It follows a formal definition of a family $\widehat{\mathcal{T}}$ of bilabelled increasing trees, which can be considered as containing all increasingly bilabelled instances of a weighted ordered tree family $\mathcal{T}$ as defined in Section~\ref{ssec:WeightedOrdTrees}.

A sequence of non-negative numbers $(\varphi_{j})_{j \ge 0}$, with $\varphi_{0} > 0$, is used to define the weight $w(T)$ of any ordered tree $T \in \mathcal{O}$ by $w(T) := \prod_{v \in T} \varphi_{\odeg(v)}$, where $v$ ranges over all vertices of $T$ and $\odeg(v)$ is the
out-degree of $v$. Furthermore, $\mathcal{L}(T)$ denotes the set of different increasing bilabellings of the tree $T$ with distinct integers of the label set $\{1, 2, \dots, 2\cdot|T|\}$, and $\ell(T) := \big|\mathcal{L}(T)\big|$ its cardinality.
Then the family $\widehat{\mathcal{T}}$ consists of all trees $T \in \mathcal{O}$ together with their
weights $w(T)$ and the set of increasing bilabellings $\mathcal{L}(T)$, i.e., one might think of triples $(T,w(T),L(T))$, with $L(T) \in \mathcal{L}(T)$ an increasing bilabelling of $T$.
For a combinatorial tree family $\widehat{\mathcal{T}}$, the total weights $T_{n} := \sum_{T \in \mathcal{O}(n)} w(T) \cdot \ell(T)$ of size-$n$ trees can be interpreted simply as the number of bilabelled increasing trees in $\widehat{\mathcal{T}}$ with $n$ nodes and label set $[2n]$.

Given a degree-weight sequence $(\varphi_{j})_{j \ge 0}$, we define the corresponding \emph{degree-weight generating function} via $\varphi(t) := \sum_{j \ge 0} \varphi_{j} t^{j}$.
Then it follows that the family $\widehat{\mathcal{T}}$ can be described by the following symbolic equation:
\begin{equation}
\label{HookBijII-eqn1}
  \widehat{\mathcal{T}} = \mathcal{Z}^{\Box} \ast \left(\mathcal{Z}^{\Box} \ast \varphi\big(\widehat{\mathcal{T}}\big)\right),
\end{equation}
where $\mathcal{Z}$ denotes the \emph{atomic class} (i.e., a single (uni)labelled node),
$\mathcal{A} \ast \mathcal{B}$ denotes the \emph{labelled product} and $\mathcal{A}^{\Box} \ast \mathcal{B}$ the \emph{boxed product} (i.e., the smallest label is constrained to lie in the $\mathcal{A}$ component) of the combinatorial classes $\mathcal{A}$ and $\mathcal{B}$, and $\varphi(\mathcal{A}) = \varphi_{0} \cdot \{\epsilon\} + \varphi_{1} \cdot \mathcal{A} + \varphi_{2} \cdot \mathcal{A}^{2} + \cdots$ denotes the class containing all \emph{labelled finite weighted sequences} of objects of $\mathcal{A}$ (i.e., each sequence of length $k$ is weighted by $\varphi_{k}$; $\epsilon$ denotes here the neutral object of size $0$), see \cite{FlaSed2009}.

\subsection{Generating functions and differential equations}
Let $T(z)$ denote the exponential generating function $T(z) := \sum_{n \ge 1} T_{n} \frac{z^{2n}}{(2n)!}$ of the number of bilabelled increasing trees in $\widehat{\mathcal{T}}$ with $n$ nodes and label set $[2n]$. Note that $T_{n}$ counts objects with $2n$ labels and the definition of $T(z)$ is thus in accordance with the symbolic description of $\widehat{\mathcal{T}}$ given in \eqref{HookBijII-eqn1}. Namely, the combinatorial construction~\eqref{HookBijII-eqn1} translates directly into an autonomous second order differential equation for the exponential generating function $T(z)$:
\begin{equation}
   \label{HookBijII-DGL1}
   T''(z) = \varphi\big(T(z)\big), \quad T(0)=0,\quad T'(0)=0.
\end{equation}

It is convenient to translate the second order equation into a first-order equation, leading to an implicit representation of $T(z)$.

\begin{prop}
\label{HookBijII-DGLProp}
The exponential generating function $T(z)$ of bilabelled increasing trees with degree-weight generating function $\varphi(t)$ satisfies the first order differential equation
\begin{equation}
\label{HookBijII-DGL2}
    T'(z) = \sqrt{2\cdot\Phi(T(z))}, \quad T(0)=0,
\end{equation}
with $\Phi(x)=\int_{0}^{x} \varphi(t)dt$. Moreover, $T=T(z)$ is given implicitly via
\[
\int_{0}^{T}\frac{dx}{\sqrt{2\cdot\Phi(x)}}=z.
\]
\end{prop}

\begin{proof}
In order to obtain the first-order differential equation we proceed in a standard way.
Multiplying \eqref{HookBijII-DGL1} with $T'(z)$ gives 
\begin{equation*}
    T'(z)\cdot T''(z) = T'(z)\varphi\big(T(z)\big),
\end{equation*}
and integrating this equation yields
\begin{equation*}
    \frac{\big(T'(z)\big)^2}{2} = \Phi(T(z)).
\end{equation*}
Consequently, we obtain 
\begin{equation*}
    T'(z) = \sqrt{2\cdot\Phi(T(z))}, \quad T(0)=0.
\end{equation*}
Note that by definition of $\Phi(x)$ and $T(0)=0$ we always have $T'(0)=\sqrt{2\Phi(0)}=0$.
Separation of variables and integration gives
\[
\int_{0}^{T}\frac{dx}{\sqrt{2\cdot\Phi(x)}}=z +C.
\]
Evaluating at $z=0$ further yields $0=0+C$ and thus shows the stated result.
\end{proof}

\subsection{Bilabelled increasing trees and hook-length formul\ae{}}
Given a tree $T$ of size $n$ with distinguishable nodes (e.g., an ordered tree or an unordered labelled tree) and the label set $\mathcal{M} = [2n]$.
When enumerating the number of increasing bilabellings of $T$
the hook-lengths of the nodes of $T$ appear naturally.
\begin{lemma}[~\cite{KubPan2012}]
\label{HookBijII-lem1}
  The number $|\mathcal{L}(T)|$ of different increasing bilabellings of a tree $T$ of size $n$ with distinguishable nodes is given as follows:
  \begin{equation*}
    |\mathcal{L}(T)| = \frac{(2n)!}{\prod_{v \in T}\left(2h_{v}(2h_{v}-1)\right)}.
  \end{equation*}
\end{lemma}
The proof of Lemma~\ref{HookBijII-lem1} can be carried out using induction; in Lemma~\ref{HookBijOutlookLem1} we will prove a generalization of this result.

As pointed out above, a family $\widehat{\mathcal{T}}$ of bilabelled increasing trees, i.e., a family of increasingly bilabelled weighted ordered trees with degree-weight sequence $(\varphi_{j})_{j \ge 0}$, consists of all triples $(T, w(T), L(T))$, with $T \in \mathcal{O}$ an ordered tree, $w(T)$ the weight of the ordered tree $T$ defined in terms of the degree-weight sequence via $w(T)=\prod_{v\in T}\varphi_{\odeg(v)}$, and $L(T) \in \mathcal{L}(T)$ an increasing bilabelling of $T$. Since the number $\ell(T) = |\mathcal{L}(T)|$ of increasing bilabellings of $T$ is given in Lemma~\ref{HookBijII-lem1}, we get that the total weight (i.e., for combinatorial tree families the number) $T_{n}$ of increasingly bilabelled size-$n$ trees in $\widehat{\mathcal{T}}$ is given as follows:
\begin{equation*}
 T_{n} = \sum_{T \in \mathcal{O}(n)} \frac{w(T) \cdot (2n)!}{\prod_{v \in T}\left(2h_{v}(2h_{v}-1)\right)}.
\end{equation*}
Thus, expressing the total weight $w(T)$ in terms of the degree-weights $\varphi_{\odeg(v)}$, this results in a hook-length formula for ordered trees.
\begin{theorem}
\label{HookBijIITheHook}
The family $\mathcal{O}$ of ordered trees satisfies the following hook-length formula:
\begin{equation*}
  \sum_{T \in \mathcal{O}(n)} \prod_{v \in T} \left(\frac{\varphi_{\odeg(v)}}{2h_{v}(2h_{v}-1)}\right) = \frac{T_{n}}{(2n)!},
\end{equation*}
where $T_n$ denote the total weights of bilabelled increasing trees with $2n$ labels and degree-weight generating function $\varphi(t)=\sum_{j\ge 0}\varphi_j t^j$.
\end{theorem}

\subsection{Unordered bilabelled increasing trees}
As a first application of Proposition~\ref{HookBijII-DGLProp} and Theorem~\ref{HookBijIITheHook} we 
rederive the earlier results from~\cite{KubPan2012} concerning the family of unordered bilabelled increasing trees with 
degree-weight generating function $\varphi(t)=e^t$.
Thus, $\Phi(x)=e^{x}-1$ and 
\[
\int_{0}^{T(z)}\frac{dx}{\sqrt{2\cdot(e^x-1)}}=z.
\]
Integration gives the equation
\[
\sqrt{2}\arctan\Big(\sqrt{e^T-1}\Big)=z,
\]
and we obtain 
\begin{equation*}
  T(z)=\ln\big(1+\tan^2(\frac{z}{\sqrt{2}})\big).
\end{equation*} 
Extracting coefficients leads to the so-called reduced tangent numbers 
\begin{equation*}
  T_n=(2n)![z^{2n}]T(z)=\tilde{E}_n,
\end{equation*}
which might be defined via the following generating function: 
\begin{equation*}
\sum_{n \ge 1} \tilde{E}_{n} \frac{z^{2n-1}}{(2n-1)!} = \sqrt{2} \tan\big(\frac{z}{\sqrt{2}}\big).
\end{equation*}
These numbers appear in the enumeration of various combinatorial objects~\cite{FoaHan2010,Pou1989} and the sequence
starts with 
\begin{equation*}
  (T_n)=(1, 1, 4, 34, 496, 11056,\dots),
\end{equation*}
compare with Figure~\ref{HookBijIIFigureUnordered}; in the OEIS they appear as \href{https://oeis.org/A002105}{A002105}.
\begin{figure}[!thb]
\includegraphics[scale=0.5]{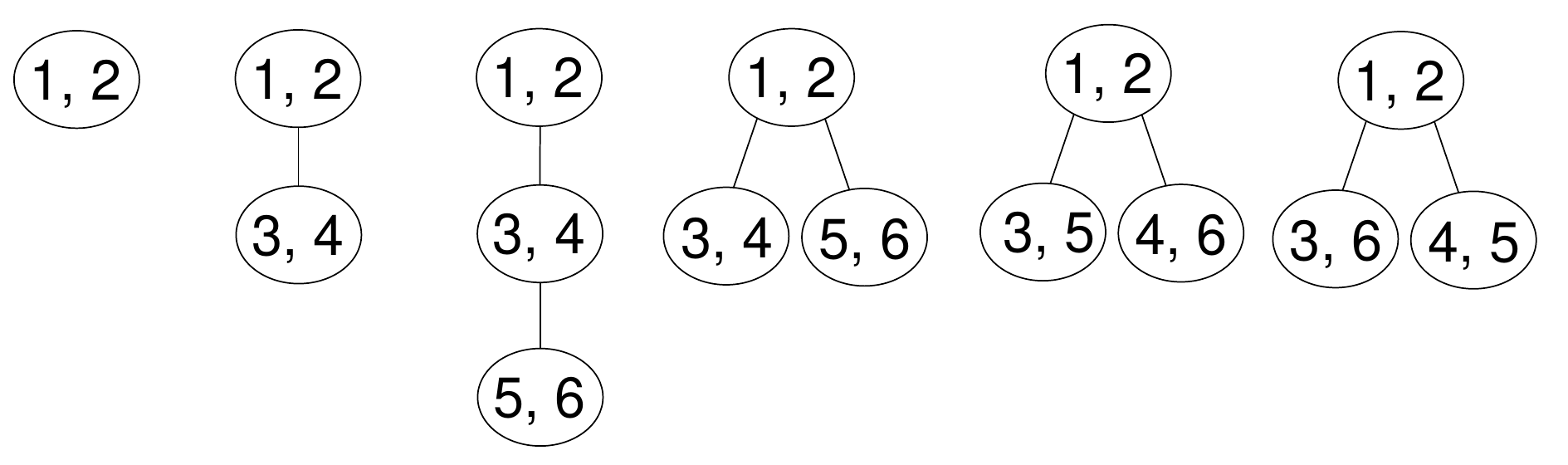}%
\caption{All unordered bilabelled increasing trees with two, four and six labels.}%
\label{HookBijIIFigureUnordered}
\end{figure}

Since $\varphi_{j}=[t^{j}]e^t = \frac{1}{j!}$, Theorem~\ref{HookBijIITheHook} gives the following hook-length formula.
\begin{coroll}
The family $\mathcal{O}$ of ordered trees satisfies the following hook-length formula:
\[
  \sum_{T\in \mathcal{O}(n)} \prod_{v \in T} \left(\frac{1}{\deg(v)! \cdot 2h_{v}(2h_{v}-1)}\right) = \frac{\tilde{E}_{n}}{(2n)!}.
\]
\end{coroll}
Note that in~\cite{KubPan2012} the above result is stated in terms of the family $\widetilde{\mathcal{U}}$ of unordered (uni)labelled trees,  thus compared to ordered labelled trees without the factor $\frac{1}{\deg(v)!}$, but with an additional factor $n!$ on the right-hand side for the number of labellings.

\section{Families of planar bilabelled increasing trees}
In this section we enumerate combinatorial models of bilabelled increasing trees, which all can be described as increasing bilabellings of certain planar rooted tree models called $d$-bundled trees, see \cite{JanKubPan2011}. As pointed out above, each enumeration result yields a corresponding hook-length formula, which is presented occasionally.

The family $\mathcal{T}$ of $d$-bundled trees, with $d$ a positive integer, can be described as follows:
The root node has $d$ positions, and at each position a (possibly empty) sequence of $d$-bundled trees is attached.
Alternatively one might think of a $d$-bundled tree as an ordered tree, where the sequence of subtrees attached to any node
in the tree is separated by $d-1$ bars into $d$ bundles. Of course, $d=1$ simply gives the family $\mathcal{O}$ of ordered trees, and in general, $d$-bundled tree families are weighted ordered trees with degree-weight generating function $\varphi(t) = \frac{1}{(1-t)^{d}}$.
In the following we state results for increasingly bilabelled $d$-bundled trees with $d \le 3$.

\subsection{Ordered bilabelled increasing trees}
Recall that the so-called error function $\erf(z)$ is defined by
\[
\erf(z)=\frac{2}{\sqrt{\pi}}\int_{0}^{z}e^{-x^2}dx,
\]
and its inverse function $\erf^{-1}(z)$ can be written as follows:
\begin{equation}
  \label{HookBijIIerfinvers}
  \erf^{-1}(z) = \sum_{k=0}^\infty\frac{c_k}{2k+1}\left (\frac{\sqrt{\pi}}{2}z\right )^{2k+1},
\end{equation}
with coefficients $c_k$ defined by $c_0=1$ and $c_k=\sum_{m=0}^{k-1}\frac{c_m c_{k-1-m}}{(m+1)(2m+1)}$, for $k>0$.

Then, the enumerative result for the family of ordered bilabelled increasing trees can be stated as follows.
\begin{theorem}
\label{HookBijIIPropPort}
The exponential generating function $T(z)$ of ordered bilabelled increasing trees with $2n$ labels and degree-weight generating function $\varphi(t)=\frac{1}{1-t}$ is given by
\begin{equation*}
  T(z) = 1-\exp\bigg(-\Big(\erf^{-1}(\frac{\sqrt{2}}{\sqrt{\pi}}z)\Big)^2\bigg) = \sqrt{\pi}\int_{0}^{\frac{z\sqrt{2}}{\sqrt{\pi}}} \erf^{-1}(x)dx.
\end{equation*}
The numbers $T_n$ are given in terms of the coefficients $c_n$ occurring in the Taylor expansion of the inverse error function \eqref{HookBijIIerfinvers} as follows:
\begin{equation*}
  T_n =\frac{(2n-2)!}{2^{n-1}} \, c_{n-1}, \qquad n \ge 1,
\end{equation*}
and they satisfy the recurrence relation
\begin{equation*}
  T_{n} = \sum_{k=1}^{n-1} \binom{2n-2}{2k} T_{k} T_{n-k}, \quad \text{for} \enspace n \ge 2, \quad \text{with} \enspace T_{1} = 1.
\end{equation*}
\end{theorem}

\begin{remark}
The sequence $(T_n)$ begins with
\begin{equation*}
  (T_n)_{n \ge 1} = (1, 1, 7, 127, 4369, 243649,\dots),
\end{equation*}
compare with Figure~\ref{HookBijIIFigurePort}; in the OEIS the numbers $T_n$ appear as \href{https://oeis.org/A002067}{A002067}.
\end{remark}
\begin{figure}[!thb]
\includegraphics[scale=0.5]{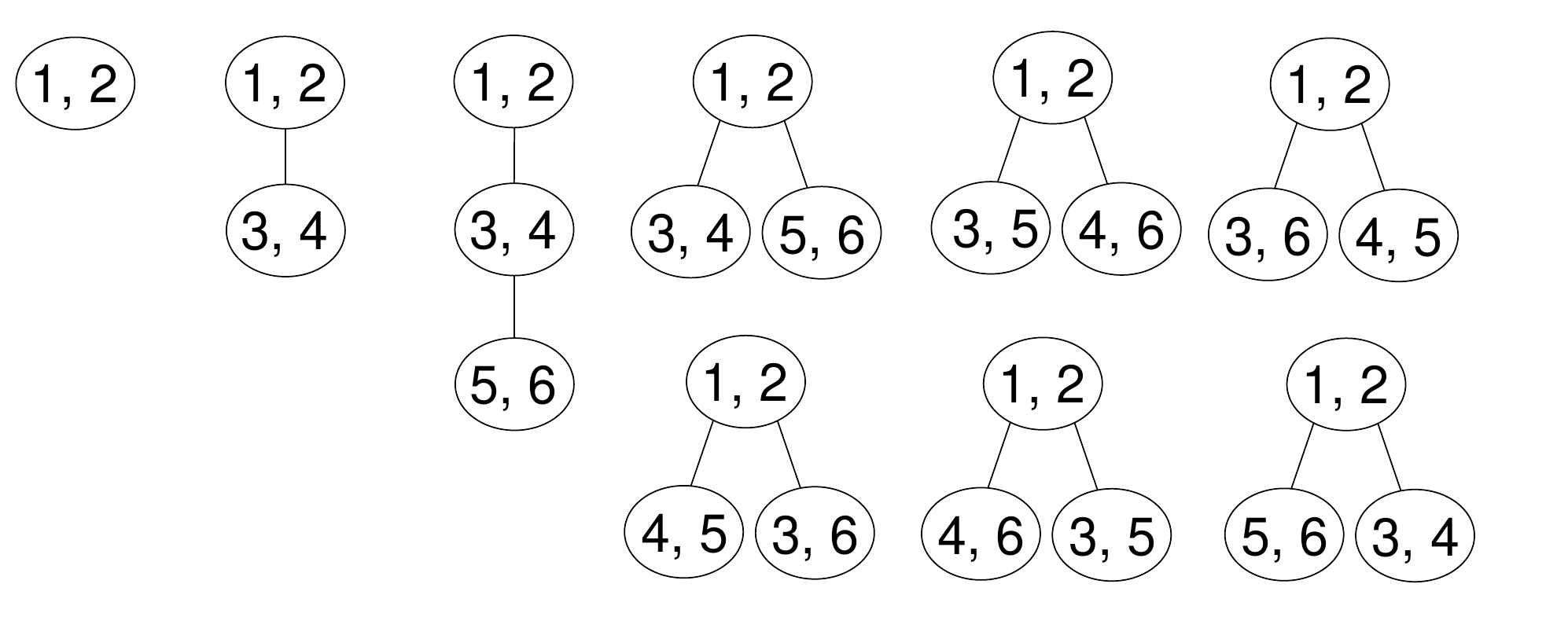}%
\caption{All ordered bilabelled increasing trees with two, four and six labels.}%
\label{HookBijIIFigurePort}
\end{figure}

\begin{proof}
In order to solve the differential equation
\begin{equation}\label{eqn:OrderedDEQ}
  T''(z) = \frac{1}{1-T(z)}, \quad T(0) = T'(0)=0,
\end{equation}
we apply Proposition~\ref{HookBijII-DGLProp} using the definitions given there.
We have
\[
\Phi(x)=\int_0^{x}\varphi(t)=\int_0^{x}\frac{1}{1-t} \, dt = L(x),
\]
where we use the shorthand notation $L(x)=-\ln(1-x)$. 
Thus, we get the equation 
\[
\int_{0}^T\frac1{\sqrt{2L(x)}} \, dx = z.
\]
An antiderivative of $\frac1{\sqrt{2L(x)}}$ is readily obtained using the substitution $x=1-e^{-u^2}$:
\[
\int_0^T\frac1{\sqrt{2L(x)}} \, dx = \frac{\sqrt{\pi}}{\sqrt{2}}\erf(\sqrt{L(T)})=z.
\]
Consequently, 
\[
\erf\big(\sqrt{-\ln(1-T(z))}\big)=\frac{\sqrt{2}}{\sqrt{\pi}}z,
\]
such that
\[
\ln(1-T(z))=-\big(\erf^{-1}\Big(\frac{\sqrt{2}}{\sqrt{\pi}}z\Big)\bigg)^2.
\]
We readily obtain the first part of the stated result by solving for $T(z)$. The antiderivative $\int \erf^{-1}(x)dx $ of the inverse error function
can be obtained using the formula for the antiderivative of an inverse function
$$
\int f^{-1}(x)dx=xf^{-1}(x)-F(f^{-1}(x))+C,\qquad F(x)=\int f(x)dx,
$$
which is easily proven using the substitution $y=f(x)$. The antiderivative of the error function
is obtained by integration by parts:
$\int\erf(x)dx=\int1\cdot \erf(x)dx=x\erf(x)+\frac1{\sqrt{\pi}}e^{-x^2}+C$.
Thus, 
\begin{equation*}
\begin{split}
\int\erf^{-1}(z)dz&=z\erf^{-1}(z)- \erf^{-1}(z)\erf(\erf^{-1}(z))\\
\quad &-\frac1{\sqrt{\pi}}\exp\Big(-\big(\erf^{-1}(z)\big)^2\Big)+C\\
&=C-\frac1{\sqrt{\pi}}\exp\Big(-\big(\erf^{-1}(z)\big)^2\Big).
\end{split}
\end{equation*}
Consequently, 
\[
\sqrt{\pi}\int_0^z\erf^{-1}(x)dx=1-\exp\Big(-\big(\erf^{-1}(z)\big)^2\Big).
\]
Thus we obtain the second part of the stated result replacing $z$ by $z\sqrt{2}/\sqrt{\pi}$. 
Concerning extracting coefficients it is beneficial to use the second expression for $T(z)$ in terms of the antiderivative and the power series expansion of $\erf^{-1}(z)$ as stated in~\eqref{HookBijIIerfinvers}. The recurrence relation for $T_{n}$ follows from \eqref{eqn:OrderedDEQ} by extracting coefficients (after multiplying with $(1-T(z))$), which completes the proof.
\end{proof}

Since the degree-weight generating function is given by $\varphi(t)=\frac{1}{1-t}=\sum_{j\ge 0}t^j$, 
it follows that $\varphi_j=[t^j]\varphi(t)=1$, for $j\ge 0$. Hence, from Theorem~\ref{HookBijIITheHook} and Proposition~\ref{HookBijIIPropPort} we obtain the following result.
\begin{coroll}
The family $\mathcal{O}$ of ordered trees satisfies the following hook-length formula:
\[
  \sum_{T\in \mathcal{O}(n)} \prod_{v \in T} \left(\frac{1}{2h_{v}(2h_{v}-1)}\right) = \frac{c_{n-1}}{n(2n-1) \, 2^n},
\]
with $c_n$ occurring as coefficients in the Taylor expansion of the inverse error function \eqref{HookBijIIerfinvers}.
\end{coroll}

\subsection{3-bundled bilabelled increasing trees}
We state the surprisingly explicit enumeration results for this tree family.
\begin{theorem}
\label{HookBijIIPropGenPort}
The exponential generating function $T(z)$ of the number $T_{n}$ of $3$-bundled bilabelled increasing trees with $2n$ labels and degree-weight generating function $\varphi(t)=\frac{1}{(1-t)^{3}}$ is given by
\begin{equation*}
  T(z)=1-\sqrt{1-z^2}.
\end{equation*}
The numbers $T_n$ are given by
\begin{equation*}
  T_n =(2n-3)!! \, (2n-1)!!.
\end{equation*}
\end{theorem}

\begin{remark}
The sequence $(T_n)$ starts with 
\begin{equation*}
  (T_n)=(1,3,45, 1575, 99225,\dots),
\end{equation*}
and in the OEIS the numbers $T_n$ appear as \href{https://oeis.org/A079484}{A079484}.
\end{remark}
\begin{figure}[!htb]
\includegraphics[scale=0.6]{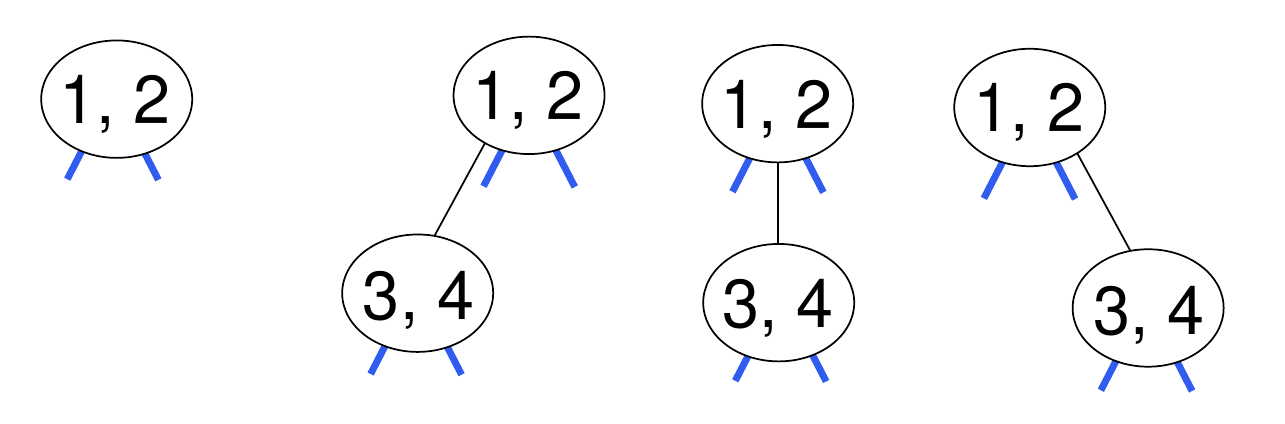}%
\caption{All 3-bundled bilabelled increasing trees with two and four labels.}%
\end{figure}

\begin{proof}
Again we apply Proposition~\ref{HookBijII-DGLProp}.
Here we get
\[
\Phi(x)=\int_0^{x}\varphi(t)=\frac1{2(1-x)^2}-\frac12.
\]
Separation of variables leads to
\[
\int_0^{T} \frac{1-x}{\sqrt{x(2-x)}} \, dx=\sqrt{T(2-T)}=z.
\]
Solving the quadratic equation yields $T(z)=1-\sqrt{1-z^2}$, and extracting coefficients, for $n \ge 1$, gives the stated explicit enumeration result:
\begin{equation*}
\begin{split}
T_n & =(2n)![z^{2n}]T(z)=(2n)![z^{n}](1-\sqrt{1-z})=(2n)!\cdot \binom{\frac12}{n}(-1)^{n-1}\\
&=\frac{(2n)!}{2^n n!}\cdot (2n-3)!! = (2n-1)!!\cdot (2n-3)!!.
\end{split}
\end{equation*}
\end{proof}

The degree-weight generating function of $3$-bundled bilabelled increasing trees is given by $\varphi(t)=\frac{1}{(1-t)^3}$, 
thus it holds that $\varphi_j=[t^j]\varphi(t)=\binom{j+2}{2}$, for $j \ge 0$. 
Consequently,
\[
\varphi_{\deg(v)}=\frac{(\odeg(v)+2)(\odeg(v)+1)}{2}.
\]
Moreover, we have
\[
 \prod_{v\in T}\varphi_{\deg(v)} = \frac{1}{2^{|T|}}\prod_{v\in T}\Big((\odeg(v)+2)(\odeg(v)+1)\Big).
\]
Since 
\[
\frac{T_{n}}{(2n)!} \cdot 2^n = \frac{(2n-3)!!}{n!},
\]
Theorem~\ref{HookBijIITheHook} and Proposition~\ref{HookBijIIPropGenPort} give the following result.

\begin{coroll}
The family $\mathcal{O}$ of ordered trees satisfies the following hook-length formula:
\[
\sum_{T\in\mathcal{O}(n)} \prod_{v\in T}\left(\frac{(\odeg(v)+2)(\odeg(v)+1)}{2h_{v}(2h_{v}-1)}\right)
= \frac{(2n-3)!!}{n!}.
\]
\end{coroll}

\subsection{2-bundled bilabelled increasing trees}
\begin{theorem}
\label{thm:2bundled}
The exponential generating function $T(z)$ of the number $T_{n}$ of $2$-bundled bilabelled increasing trees with $2n$ labels and degree-weight generating function $\varphi(t)=\frac{1}{(1-t)^{2}}$ is given implicitly via
$$2\Big(\arcsin(\sqrt{T})+\sqrt{T}\sqrt{1-T}\Big)^2=z^2.$$
The numbers $T_n$ satisfy for $n \ge 2$ the following recurrence (with $T_{1}=1$):
\begin{equation*}
  T_{n} = 2 \sum_{k=1}^{n-1} \binom{2n-2}{2k} T_{k} T_{n-k} - \sum_{j+k+\ell=n-1} \binom{2n-2}{2j, 2k, 2\ell} T_{j} T_{k} T_{\ell+1}.
\end{equation*}
Let the sequence $(x_k)$ be defined by $x_k=k!\frac{k\binom{2k}k }{4^{k}(2k+1)}$.
The number $T_n$ of trees with $2n$ labels $n>1$ can be given in terms of the Bell polynomials as follows:
$$T_n =\frac{(2n)!}{n}\frac{1}{8^n}\sum_{m=1}^{n-1}\binom{2n-1+m}{m}B_{n-1,m}(x_1,\dots,x_{n-m}).$$
\end{theorem}

\begin{remark}
The sequence $(T_n)$ begins with
\begin{equation*}
  (T_n)_{n \ge 1} = (1, 2, 22, 584, 28384, 2190128,\dots);
\end{equation*}
in the OEIS the numbers $T_n$ appear as \href{https://oeis.org/A120419}{A120419}, but without giving a combinatorial interpretation of this enumeration sequence. Thus we are able to present such one.
\end{remark}

\begin{proof}
We apply Proposition~\ref{HookBijII-DGLProp} to the differential equation
\begin{equation}\label{eqn:2bundledDEQ2}
  T''(z) = \frac{1}{(1-T(z))^{2}}, \quad T(0) = T'(0)=0,
\end{equation}
and take into account
\[
\Phi(x)=\int_0^{x}\varphi(t)=\int_0^{x}\frac{1}{(1-t)^{2}} \, dt = \frac{x}{1-x}.
\]
In order to solve
\[
\int_0^{T} \frac{\sqrt{1-x}}{\sqrt{2x}}dx=z,
\]
we obtain first an antiderivative of $\frac{\sqrt{1-x}}{\sqrt{2x}}$. 
Standard substitutions $x=u^2$ and $u=\sin(v)$ give
\[
\int_0^{T} \frac{\sqrt{1-x}}{\sqrt{2x}}dx=\sqrt{2}\big(\arcsin(\sqrt{T})+\sqrt{T}\sqrt{1-T}\big),
\]
such that $T=T(z)$ is defined by the equation
\[
\sqrt{2}\big(\arcsin(\sqrt{T})+\sqrt{T}\sqrt{1-T}\big)=z.
\]
We know a priori that $T(z)$ is a power series in $z^2$. Hence, we square the equation to obtain
\[
2\Big(\arcsin(\sqrt{T})+\sqrt{T}\sqrt{1-T}\Big)^2=z^2.
\]
This proves the first part of the stated result. The recurrence for the numbers $T_{n}$ can be obtained in a straightforward way by extracting coefficients from the equation
\begin{equation*}
  T''(z) = 1 + 2 T(z) T''(z) - T^{2}(z) T''(z),
\end{equation*}
which follows immediately from \eqref{eqn:2bundledDEQ2}.
In order to obtain the expression for the number $T_n$ we can use Lagrange's inversion formula. 
Let $Z=z^2$ and $\phi(w)$ be given by 
$$
\phi(w)=\frac{w}{2\big(\arcsin(\sqrt{w})+\sqrt{w}\sqrt{1-w}\big)^2},
$$
such that 
\[
Z=\frac{T}{\phi(T)}.
\]
Consequently, $T(Z)$ is the inverse function of $\frac{w}{\phi(w)}$ and its coefficients can be obtained as follows:
\[
T_n=(2n)![z^{2n}]T(z)=(2n)![Z^n]T(Z)=\frac{(2n)!}{n} [w^{n-1}] \phi(w)^n.
\]
In order to extract coefficients we use the fact that
\[
\arcsin(\sqrt{w})=w^{\frac12}\sum_{k\ge 0}\frac{\binom{2k}k }{4^k(2k+1)}w^{k},\qquad
\sqrt{w}\sqrt{1-w}=w^{\frac12}\Big(1-\sum_{k\ge 1}\binom{2k}k \frac{w^k}{4^k}\Big),
\]
such that $\phi(w)=\frac{1}{8(1-\vartheta(w))^2}$, with
\begin{equation}
\label{HookBijTwoBundled}
\vartheta(w)=\sum_{k\ge 1}\vartheta_k w^k=\sum_{k\ge 1}\frac{k\binom{2k}k }{4^{k}(2k+1)} w^k.
\end{equation}
Hence, we obtain the formal power series 
\[
\phi^n(w)=\frac{1}{8^n(1-\vartheta(w))^{2n}}=\frac{1}{8^n}\Big(1+\sum_{k\ge 1}\phi_{n,k}w^k\Big),
\]
with $\phi_{n,k}$ for $k\ge 1$ given by
\[
\phi_{n,k}=\sum_{m=1}^k\binom{2n-1+m}{m}\sum_{\substack{j_1+\dots+j_k=m\\\sum_{i=1}^k i j_i=k}}
\binom{m}{j_1,\dots,j_m}\prod_{\ell=1}^{k}\bigg(\frac{\ell\binom{2\ell}{\ell}}{4^{\ell}(2\ell+1)}\bigg)^{j_\ell}.
\]
We can use Bell polynomials $B_{k,m}(x_1,\dots,x_{k-m+1})$
evaluated at $x_k=k! \vartheta_k$ to get the equivalent expression
\[
\phi_{n,k}=\sum_{m=1}^k\binom{2n-1+m}{m}B_{k,m}(x_1,\dots,x_{k-m+1}).
\]
Consequently,
\begin{equation*}
\begin{split}
T_n&=\frac{(2n)!}{n} [w^{n-1}] \phi(w)^n=\frac{(2n)!}{n}\frac{1}{8^n}\phi_{n,n-1}\\
&=\frac{(2n)!}{n}\frac{1}{8^n}\sum_{m=1}^{n-1}\binom{2n-1+m}{m}B_{n-1,m}(x_1,\dots,x_{n-m}).
\end{split}
\end{equation*}
This proves the stated result about $T_n$.
\end{proof}

\section{Elliptic families of bilabelled increasing trees}

In the following we present several \textit{elliptic families} of bilabelled increasing trees: we discuss in detail increasing bilabellings of strict-binary trees with degree-weight generating function $\varphi(t)=1+t^2$ and of unordered even-degree trees with degree-weight generating function $\varphi(t)=\cosh(t)$. Moreover, we present a general approach to uncover the elliptic nature of arbitrary families of binary and ternary bilabelled increasing trees.

An elliptic function is a function that is meromorphic in the whole complex plane and that is doubly periodic. 
A standard way of presenting the theory of elliptic functions is the one proposed by Eisenstein and Weierstrass, 
where elliptic functions are defined as sums of rational functions taken over lattices. 
Given a pair of complex numbers $\omega_1$ and $\omega_2$ generating a lattice $\Omega=\Z\omega_1 + \Z\omega_1$, the Weierstrass-$\wp$ function, see \cite{KoecherKrieg}, $\wp(z)=\wp(z\mid\omega_1,\omega_2)$ is 
by construction a double-periodic function with periods $\omega_1,\omega_2$ defined as
\begin{equation}
  \label{eqn:WPmerosum}
  \wp(z\mid \omega_1,\omega_2)=\frac{1}{z^2}+\sum_{0\neq \omega \in \Omega}\left(\frac{1}{(z-\omega)^2}-\frac{1}{\omega^2}\right).
\end{equation}
It satisfies the first-order differential equation 
\begin{equation}
\label{WeierstrassP}
(\wp'(z))^2 = 4(\wp(z))^3-g_2\wp(z)-g_3, 
\end{equation}
and also the second-order differential equation
\[
2\wp''(z)=12(\wp(z))^2-g_2.
\]
Here $g_2$ and $g_3$ are the so-called Weierstrass-invariants defined by 
\begin{equation*}
\begin{split}
g_2&= 60\sum_{0\neq \omega \in \Omega} \frac1{\omega^{4}},
\qquad g_3=140\sum_{0\neq \omega \in \Omega} \frac1{\omega^{6}},
\end{split}
\end{equation*}
which can be used to specify the Weierstrass-$\wp$ function alternatively in terms of its invariants: $\wp(z)=\wp(z;g_2,g_3)$;
in this case we use the notation $\omega_1(g_2,g_3)$ and $\omega_2(g_2,g_3)$ for the resulting periods.
All solutions of the differential equation~\eqref{WeierstrassP} have the form $\wp(z+C)$, where the constant $C$ depends on the initial value. 

\smallskip

We also introduce the lemniscate sine function $\lemsine(z)$, see~\cite{KoecherKrieg}, which is an elliptic function defined as the inverse of the Fagnano elliptic integral:
\[
\lemsine(z) = s,\qquad z=\int_{0}^{s}\frac{1}{\sqrt{1-t^4}}dt.
\]
The coefficients $S_{n}$ in the series expansion of the lemniscate sine function:
\begin{equation}\label{eqn:lemsine_coeff}
  \lemsine(z) = \sum_{n \ge 1} S_{n} \frac{z^{n}}{n!},
\end{equation}
appear in the OEIS as \href{https://oeis.org/A104203}{A104203}. 
Moreover, $\varpi$ denotes the so-called lemniscate constant:
\begin{equation}
\label{HookBijIILemniKonst}
\varpi = 2\int_{0}^{1}\frac{1}{\sqrt{1-t^4}}dt = \frac{\Gamma(\frac{1}{4})^2 }{ 2 \sqrt{2\pi}}.
\end{equation}

\subsection{Strict-binary bilabelled increasing trees\label{StrictBinary}}
\begin{theorem}
The exponential generating function $T(z)$ of the number $T_{n}$ of strict-binary bilabelled increasing trees with $2n$ labels and degree-weight generating function $\varphi(t)=1+t^2$ is given in terms of the Weier\-strass-$\wp$ function $\wp(z; g_2,g_3)$ as follows:
\begin{equation*}
T(z)=\frac{6}{\sqrt{3}} \cdot \wp(3^{-\frac14}z+\varpi; -1 , 0).
\end{equation*}
Alternatively, $T(z)$ can be expressed in terms of the square of the lemniscate sine function $\lemsine(z)$:
\[
T(z)=\sqrt{3} i \, \lemsine^2\Big(\frac{z}{3^{\frac{1}{4}} (1+i)}\Big).
\]

\smallskip

The numbers $T_n$ satisfy the recurrence relation
\[
T_n=\sum_{k=1}^{n-2}\binom{2n-2}{2k}T_k T_{n-1-k},\quad \text{for} \enspace n \ge 2, \quad \text{with} \enspace T_1=1.
\]
Moreover, they can be expressed as a lattice sum in the following way:
\[
T_n=\frac{(2n+1)! \, 2^{3n+4}\pi^{n+1}}{3^{\frac{n-1}2}\Gamma^{4n+4}(\frac14)}
\sum_{n_1,n_2\in\Z}\frac{1}{(1+n_1+n_2+i(n_1-n_2))^{2n+2}}.
\]
\end{theorem}

\begin{remark}
The sequence $(T_n)$ starts with 
\begin{equation*}
 (T_n)=(1,0,6,0,336,0,77616,0,50916096,0,\dots),
\end{equation*}
and in the OEIS it appears as \href{https://oeis.org/A144849}{A144849}. 
The relation of these numbers to the (square of the) lemniscate sine function $\lemsine(z)$
has been observed before by Michael Somos (see \href{https://oeis.org/A144849}{A144849}).
\end{remark}

\begin{proof}
By Proposition~\ref{HookBijII-DGLProp} the exponential generating function $T(z)$ satisfies
\begin{equation}
   \label{HookBijII-eqnStrict1}
   T''(z) = 1+T^2(z), \quad T(0)=0,\quad T'(0)=0.
\end{equation}
Thus, by extracting coefficients we obtain directly the stated recurrence relation:
\[
T_n=(2n-2)![z^{2n-2}]T(z)=(2n-2)!\sum_{k=1}^{n-1}[z^{2k}]T(z)[z^{2n-2-2k}]T(z).
\]
Let $\Phi(x)=\int_0^{x}(1+t^2)=x+\frac{x^3}3$.
The relation to the Weierstrass-$\wp$ function $\wp(z; g_2,g_3)$ is a direct consequence of the first order equation:
\begin{equation}
  \label{eqn:TzElliptDEQ}
  (T'(z))^2=2\Phi\big(T(z)\big)=\frac{2}3T^3(z)+2T(z),\quad T(0)=0,
\end{equation}
and we obtain that
\[
T(z)=C_1\cdot \wp(z+C_2;g_2,g_3).
\]
Obviously, $g_3=0$. In order to identify $g_2$ we compare \eqref{WeierstrassP} and \eqref{eqn:TzElliptDEQ}, which gives
\begin{equation*}
\begin{split}
(T'(z))^2&=C_1^2 (\wp'(z+C_2;g_2,0))^2\\
&=\frac{2}3\cdot C_1^3 \wp^3(z+C_2;g_2,0)+2C_1\wp(z+C_2;g_2,0)\\
&=4 C_{1}^{2} \wp^3(z+C_2;g_2,0)- C_{1}^{2} g_2\wp(z+C_2;g_2,0).
\end{split}
\end{equation*}
Consequently, we obtain the pair of relations
\[
\frac{2}{3} C_{1}^{3} = 4 C_{1}^{2}, \qquad C_{1}^{2} \, g_2 = -2 C_1,
\]
such that $C_1=6$ and $g_2=-\frac13$ leading to $T(z)= 6\cdot \wp(z+C_2; -\frac13, 0)$.

The case $g_2<0$ and $g_3=0$ can be reduced to the so-called pseudo-lemniscatic case with $g_2=-1$ and $g_3=0$ for which the periods are known, see \cite{AbramowitzStegun1972}.
By the homogeneity relation
\[
\wp(z;g_2, 0)=|g_2|^{\frac12}\wp(z\cdot|g_2|^{\frac14};-1, 0)
\]
we obtain
\[
T(z)=\frac{6}{\sqrt{3}} \cdot \wp(3^{-\frac14}(z+C_2); -1, 0).
\]
Hence we know, see \cite[page 662]{AbramowitzStegun1972}, that $\omega_1=(1+i)\varpi$ and $\omega_2=(1-i)\varpi$,
with $\varpi$ denoting the lemniscate constant~\eqref{HookBijIILemniKonst}.

It remains to adapt the constant $C_2\in\C$, where we use that $T(z)$ has a double zero at $z=0$. In the following we use the standard notation $e_k=\wp(\omega_k/2)$, for $k=1, 2, 3$, with $\omega_3=\omega_1+\omega_2$ the sum of the periods. 
It is known, see \cite[p.~29]{KoecherKrieg}, that $\wp(z)-e_k$ has a double zero at $z=\omega_k/2$ for $k=1,2,3$. 
The values $e_k$ are alternatively defined as the roots of a third order polynomial: $4X^3-g_2X-g_3=4(X-e_1)(X-e_2)(X-e_3)$. 
In our case $g_3=0$ and $-g_2=1>0$, such that $e_1,e_2\in i\cdot \R$ with $\overline{e_2}=e_1$, and $e_3=0$. 
Consequently, $\wp(z)-e_3=\wp(z)$ has a double zero at 
\begin{equation}
\label{HookBijII-eqnStrict2}
\omega_3/2=(\omega_1+\omega_2)/2=\varpi,
\end{equation}
which implies $C_2=3^{\frac14}\varpi$. This proves the stated result.

\medskip

The lattice sum expression for $T_n$, $n\ge 1$, is now readily obtained using
\begin{equation}
[z^{2n}]\frac1{(z-\omega)^2}=[z^{2n}]\frac1{\omega^2(1-\frac{z}{\omega})^2}=\frac{2n+1}{\omega^{2n+2}},
\label{HookBijIIExtraction}
\end{equation}
and taking into account the periods $\omega_1=(1+i)\varpi$ and $\omega_2=(1-i)\varpi$.
First, we obtain
\begin{equation*}
\begin{split}
T_n&=(2n)!\cdot [z^{2n}]\frac{6}{3^{\frac12}}\cdot \wp(3^{-\frac14}z+\varpi; -1 , 0)\\
&=\frac{2(2n)!}{3^{\frac{n-1}{2}}}\cdot [t^{2n}]\wp(t+\varpi\mid  \omega_1 ,\omega_2).
\end{split}
\end{equation*}
Extracting coefficients from \eqref{eqn:WPmerosum} using~\eqref{HookBijIIExtraction} and replacing the lemniscate constant $\varpi$ by its explicit expression~\eqref{HookBijIILemniKonst} yields the enumeration formula for $T_{n}$.

\medskip

The relation to the lemniscate sine function could be obtained from the expression for the Weierstrass-$\wp$ function and its relation 
to the Jacobi elliptic function. However, we directly show that $g(z)=\lemsine^2(z)$ satisfies the differential equation
\[
g''(z)=2-6g^2(z),\quad g(0)=g'(0)=0.
\]
Since the lemniscate sine function has a power series expansion around $z=0$ with $\lemsine(0)=0$ and $\lemsine'(0)\neq 0$, 
$g(z)=\lemsine^2(z)$ satisfies $g(0)=g'(0)=0$. By its definition the derivative of the lemniscate sine function is given
\begin{equation}\label{eqn:LemSineDEQ}
\lemsine'(z)=\bigg(\frac{1}{\sqrt{1-\lemsine^4(z)}}\bigg)^{-1}=\sqrt{1-\lemsine^4(z)}.
\end{equation}
Consequently, 
\[
(\lemsine'(z))^2=1-\lemsine^4(z),
\]
and thus 
\[
2\lemsine'(z)\lemsine''(z)=-4\lemsine'(z)\lemsine^3(z),\quad \text{such that} \quad 
\lemsine''(z)=-2\lemsine^3(z).
\]
Hence, 
\begin{align*}
g''(z) & = \big(\lemsine^2(z))''
=2(\lemsine'(z))^2+2\lemsine''(z)\lemsine(z)
= 2-2\lemsine^4(z)-4\lemsine^4(z)\\
& = 2-6g^2(z),
\end{align*}
which proves that $g(z)=\lemsine^2(z)$ satisfies the stated differential equation. 
Setting $T(z)=a\cdot g(bz)$, $a,b\in\C$, leads to the system of equations
\begin{equation*}
2ab^2=1,\quad -\frac{6b^2}{a}=1. 
\end{equation*}
This system is readily solved and we obtain the stated result.
\end{proof}

The family $\mathcal{S}$ of strict-binary trees corresponds to weighted ordered trees with degree-weight generating function $\varphi(t)=1+t^2$ and, according to Theorem~\ref{HookBijIITheHook}, the enumeration result for the number $T_{n}$ of strict-binary bilabelled increasing trees of size $n$ can be translated into a hook-length formula for ordered trees $\mathcal{O}$. However, we prefer to state this formula directly in terms of the family $\mathcal{S}$.

\begin{coroll}
\label{HookBijCorollElli}
The family $\mathcal{S}$ of strict-binary trees satisfies the following hook-length formula:
\begin{align*}
\sum_{T\in\mathcal{S}(n)}  \frac{1}{\prod_{v \in T}\left(2h_{v}(2h_{v}-1)\right)}
&= 
\frac{(2n+1) \, 2^{3n+4}\pi^{n+1}}{3^{\frac{n-1}2}\Gamma^{4n+4}(\frac14)}\\
&\quad \times \sum_{n_1,n_2\in\Z}\frac{1}{(1+n_1+n_2+i(n_1-n_2))^{2n+2}}.
\end{align*}
\end{coroll}

\subsection{Unordered even-degree bilabelled increasing trees}
\begin{theorem}
The derivative of the exponential generating function $T(z)$ of the number $T_{n}$ of unordered even-degree bilabelled increasing trees with $2n$ labels and degree-weight generating function $\varphi(t)=\cosh(t)$ is given in terms of the lemniscate sine function as follows:
\begin{equation*}
  T'(z) = (1-i) \, \lemsine\!\Big(\frac{(1+i) z}{2}\Big).
\end{equation*}
The numbers $T_n$ are given in terms of the coefficients $S_n$ occurring in the Taylor expansion of the 
lemniscate sine function \eqref{eqn:lemsine_coeff} via
\begin{equation*}
  T_n = \left(\frac{i}{2}\right)^{n-1} \, S_{2n-1}, \qquad n \ge 1,
\end{equation*}
and they satisfy the recurrence relation
\begin{equation*}
  T_{n+2} = \frac{1}{2} \sum_{j+k+\ell=n-1} \binom{2n+1}{2j+1, \, 2k+1, \, 2\ell+1} T_{j+1} T_{k+1} T_{\ell+1}, \quad \text{for} \enspace n \ge 0,
\end{equation*}
with $T_{1} = 1$.
\end{theorem}

\begin{remark}
The sequence $(T_n)$ starts with 
\[
(T_n)=(1,0,3,0,189,0,68607,0,\dots);
\]
currently the sequence itself does not appear in the OEIS, but the closely related sequence $(S_{n})$ of the coefficients of the lemniscate sine function can be found as \href{https://oeis.org/A104203}{A104203}.
\end{remark}

\begin{proof}
The generating function $T(z)$ satisfies the differential equation
\begin{equation}\label{eqn:EvenDegDEQ}
  T''(z) = \cosh(T(z))
\end{equation}
with $T(0) = T'(0) = 0$; thus $T''(0) = 1$. Multiplying by $T'(z)$ and integrating the resulting equation yields
\begin{equation*}
  \frac{T'(z)^{2}}{2} = \sinh(T(z)) + \widetilde{C},
\end{equation*}
with a certain constant $\widetilde{C}$. Due to the initial conditions we get $\widetilde{C} = 0$, and thus
\begin{equation*}
  T'(z)^{2} = 2 \sinh(T(z)).
\end{equation*}
Plugging this result into the derivative of \eqref{eqn:EvenDegDEQ}, we get
\begin{equation*}
  T'''(z) = \sinh(T(z)) T'(z) = \frac{T'(z)^{3}}{2}.
\end{equation*}
Let us denote $U(z) := T'(z)$; then $U(z)$ satisfies the differential equation
\begin{equation}\label{eqn:EvenDegDEQ2}
  U''(z) = \frac{U^{3}(z)}{2},
\end{equation}
with $U(0) = 0$ and $U'(0)=1$. Multiplying \eqref{eqn:EvenDegDEQ2} by $U'(z)$ and integrating the resulting equation gives
\begin{equation*}
  U'(z)^{2} = \frac{U(z)^{4}}{4} + C,
\end{equation*}
and due to the initial conditions the constant is given by $C=1$. Thus, $U(z)$ satisfies the differential equation
\begin{equation}\label{eqn:EvenDegDEQ3}
  U'(z)^{2} = \frac{U(z)^{4}}{4} + 1, \qquad U(0) = 0.
\end{equation}
Comparing \eqref{eqn:EvenDegDEQ3} with the differential equation \eqref{eqn:LemSineDEQ} for the lemniscate sine function, shows that $U(z)$ can be written in the form
\begin{equation*}
  U(z) = a \lemsine(b z),
\end{equation*}
with certain constants $a$, $b$. Differentiating yields
\begin{equation*}
  (\lemsine'(b z))^{2} = \frac{1}{a^{2} b^{2}} + \frac{a^{2}}{4 b^{2}} \lemsine(b z)^{4},
\end{equation*}
which gives the systems of equations
\begin{equation*}
  1 = \frac{1}{a^{2} b^{2}}, \quad -1=\frac{a^{2}}{4 b^{2}}.
\end{equation*}
Solving this system, immediately leads to the stated result expressing $T'(z) = U(z)$ in terms of $\lemsine(z)$. Furthermore, extracting coefficients gives the relation between $T_{n}$ and $S_{n}$:
\begin{align*}
  T_{n} & = (2n-1)! [z^{2n-1}] T'(z) = (2n-1)! [z^{2n-1}] (1-i) \, \lemsine\!\Big(\frac{(1+i)z}{2}\Big)\\
	& = (1-i) \Big(\frac{1+i}{2}\Big)^{2n-1} S_{2n-1} = \Big(\frac{i}{2}\Big)^{n-1} S_{2n-1}.
\end{align*}
Moreover, the recurrence relation for $T_{n}$ can be obtained easily by extracting coefficients from the differential equation \eqref{eqn:EvenDegDEQ2}.
\end{proof}

\subsection{Binary and ternary bilabelled increasing trees}
The relation between bilabelled increasing tree families and ellip\-tic functions
extends to more general degree-weight generating functions. In general, by the theory of elliptic functions, all binary bilabelled increasing trees families with degree-weight generating functions $\varphi(t)$ of the form  $\varphi(t)=\varphi_2t^2+\varphi_1t + \varphi_0$, with $\varphi_0,\varphi_2>0$ and $\varphi_1\ge 0$, can be represented in terms of the Weierstrass-$\wp$ function.
We have $\Phi(t)=\frac{\varphi_2}3t^3+\frac{\varphi_1}2t^2 + \varphi_0t$ and obtain by Proposition~\ref{HookBijII-DGLProp}
the differential equation 
\[
\big(T'(z)\big)^2=\frac{2\varphi_2}{3} T^3(z) + \varphi_1 T^2(z) + 2\varphi_0 T(z). 
\]
We can reduce the equation to a depressed cubic using $T(z)=f(z)-\frac{\varphi_1}{2\varphi_2}$, and then identify the invariants $g_2,g_3$ similar to the analysis of the previous elliptic tree families. 
We obtain the result
\[
T(z)=\frac{6}{\varphi_2}\wp(z+C;g_2,g_3)-\frac{\varphi_1}{2\varphi_2},
\]
with invariants $g_2,g_3$ and constant $C$ determined by 
\[
g_2=-\frac13 \varphi_0 \varphi_2 +\frac1{12}\varphi_1^2,\quad 
g_3=-\frac1{216}\varphi_1^3+\frac1{36}\varphi_0\varphi_1\varphi_2,\quad
\wp(C; g_2, g_3)=\frac{\varphi_1}{12}.
\]

In particular, for binary bilabelled increasing trees with $\varphi(t)=(1+t)^2=1+2t+t^2$, the sequence $(T_n)$ of the number of trees starts with
\begin{equation*}
  (T_n)=(1,2,10,80,1000,17600,418000,\dots);
\end{equation*}
in the OEIS these numbers appear as \href{https://oeis.org/A063902}{A063902}. They satisfy the recurrence relation
\begin{equation*}
  T_{n} = 2 T_{n-1} + \sum_{k=1}^{n-2} \binom{2n-2}{2k} T_{k} T_{n-1-k}, \quad \text{for} \enspace n \ge 2, \quad \text{with} \enspace T_{1} = 1,
\end{equation*}
which can be shown easily by extracting coefficients from the second order differential equation for $T(z)$.

\medskip

Moreover, all ternary bilabelled increasing tree families with degree-weight generating functions $\varphi(t)$ of the form 
$\varphi(t)=\varphi_3t^3+\varphi_2t^2+\varphi_1t + \varphi_0$, with $\varphi_0,\varphi_3>0$ and $\varphi_1,\varphi_2\ge 0$, can be expressed as a reciprocal of the Weierstrass-$\wp$ function.
First, we obtain 
$$\Phi(t)=\sum_{k=1}^{4}\Phi_k t^k=\frac{\varphi_3}4t^4+\frac{\varphi_2}3t^3+\frac{\varphi_1}2t^2 + \varphi_0t.$$ 
The polynomial has the root zero and by Descartes' rule of signs also one negative root, which we denote by $t_0\in (-\infty,0)$. 
We can use a classical reduction, see~\cite[page 2]{KoecherKrieg}, of the differential equation
\[
\big(T'(z)\big)^2=2\Phi(T(z))=\frac{\varphi_3}2T^4(z)+\frac{2\varphi_2}3T(z)^3+\varphi_1T(z)^2 + 2\varphi_0T(z), \quad T(0)=0,
\]
to a differential equation for $R(z)=\frac{1}{T(z)-t_0}$:
\[
\big(R'(z)\big)^2=q_3R^3(z)+q_2R^2(z)+q_1R(z)+q_0,\quad R(0)=-\frac{1}{t_0},
\]
with 
\begin{equation*}
  q(t)=\sum_{k=0}^{3}q_k t^k=\sum_{k=0}^{3} 2\frac{\Phi^{(4-k)}(t_0)}{(4-k)!} t^k.
\end{equation*}
Thus, for a representation of $R(z)$ as a Weierstrass-$\wp$ function (and a representation of $T(z)$ as a reciprocal of $\wp$, respectively) one may proceed along the lines as carried out before for families of binary bilabelled increasing trees. We omit these more involved explicit computations and the corresponding results.

\section{A reverse engineering approach\label{sec:ReverseEngineering}}

We have observed interesting connections of the generating functions of bilabelled increasing trees to special functions and elliptic functions, respectively.
However, it turns out that besides the family of increasing bilabelled 3-bundled trees there does not seem to exist simple closed-from expressions for the numbers $T_n$
associated to the most common degree-weight generating functions $\varphi(t)=e^t$, $\varphi(t)=1/(1-t)^\alpha$ with $\alpha>0$, or $\varphi(t)=(1+t)^d$, $d\ge 2$. 
We can use a different approach in order to obtain simple closed form solutions for the number of bilabelled increasing trees with $n$ nodes and label set $[2n]$. Instead of choosing the degree-weight generating function $\varphi(t)$ and studying the differential equations~\eqref{HookBijII-DGL1} and~\eqref{HookBijII-DGL2},
we \emph{select first}\footnote{We follow the maxim of Jacobi: ``man muss immer umkehren'', which means ``Invert, always invert''.} the exponential generating function $T(z)$, and determine the arising degree-weight generating function afterwards.

\subsection{The general procedure}
\begin{enumerate}
	\item Select an exponential generating function $T=T(z)=\sum_{n\ge 1}T_n\frac{z^{2n}}{(2n)!}$ with \emph{a priori given} numbers $T_n$.
	We assume that $T(z)=f(z^2)$, with $f$ invertible and twice differentiable.
	
  \medskip
	
	\item Express $z=z(T)$ using the inverse function of $f$: $z=\sqrt{f^{-1}(T)}$. 
	
	\medskip
	
		\item Rewrite the differential equation $T''(z)=\varphi(T(z))$ in terms of $f$ and $z=z(T)$:
		using $T''(z)=4z^2f''(z^2)+2f'(z^2)$ we get
		\begin{equation*}
	    \varphi\big(T\big)=T''(z(T))= 4f^{-1}(T)f''\left(f^{-1}(T)\right)+2f'\left(f^{-1}(T)\right).
		\end{equation*}
	
	\medskip
		
	\item Test if $\varphi(T) = \sum_{j \ge 0} \varphi_{j} T^{j}$ has non-negative coefficients, with 
	$\varphi_0>0$, $\varphi_\ell\ge 0$, for all $\ell \in \N$. 
	If yes, we have found a combinatorial family $\mathcal{T}$ of bilabelled increasing trees satisfying the formal equation
	\[
	\mathcal{T} = \mathcal{Z}^{\Box} \ast \left(\mathcal{Z}^{\Box} \ast \varphi\big(\mathcal{T}\big)\right),
	\]
	with degree-weight generating function given via
	\begin{equation*}
	  \varphi\big(T\big)=4f^{-1}(T)f''\left(f^{-1}(T)\right)+2f'\left(f^{-1}(T)\right).
	\end{equation*}
	
  \medskip
		
	\item Interpret the tree family $\mathcal{T}$ with degree-weight generating function $\varphi(t)$ in terms of weighted ordered trees $\mathcal{O}$, which, according to Theorem~\ref{HookBijIITheHook}, yields the hook-length formula
	\begin{equation*}
    \sum_{T \in \mathcal{O}(n)} \prod_{v \in T} \left(\frac{\varphi_{\odeg(v)}}{2h_{v}(2h_{v}-1)}\right) 
	  = \frac{T_{n}}{(2n)!}.
  \end{equation*}
\end{enumerate}

\begin{example}
Let 
\[
T(z)=C\cdot(1-(1-Az^2)^B),
\]
with $A,B,C\in\R$ constant. Extracting coefficients leads to $$T_n=(2n)![z^{2n}]T(z)=(2n)!\cdot (-C)\cdot (-A)^n\binom{B}{n}.$$ 
For the two cases (i) $B<0$, such that $C<0$ and $A>0$, and (ii) $0<B<1$ with $C>0$ and $A>0$, the total weights $T_n$ are positive, $T_n> 0$, for $n\ge 1$.

\smallskip

Then, $T(z)=f(z^2)$ with $f(z)=C\cdot(1-(1-Az)^B)$. Consequently, 
\[
z=z(T)=\sqrt{f^{-1}(T)}=\sqrt{\frac{1-\Big(1-\frac{T}{C}\Big)^{\frac1B}}{A}}.
\]
The second derivative of $T(z)$ is given by
\[
T''(z)=4A^2BC(1-B)(1-Az^2)^{B-2}+ 2ABC(1-Az^2)^{B-1}.
\]
Consequently, due to the demand $\varphi(T)=T''(z(T))$ we obtain the result
\[
\varphi(T)=4ABC(1-B)\big(1-\frac{T}{C}\big)^{1-\frac2B}+2ABC(2B-1)\big(1-\frac{T}{C}\big)^{1-\frac1B}.
\]

In case (i) $A>0$, $B<0$, $C<0$, the function $\varphi(T)$ satisfies
\[
\varphi(T)=\sum_{j \ge 0}\varphi_j T^j =
2ABC\sum_{j \ge 0}\frac{(2-2B)\binom{1-\frac2B}{j}-(1-2B)\binom{1-\frac1B}{j}}{(-C)^j}T^j.
\]
From the condition $\varphi_j \ge 0$, with 
\begin{equation}
\label{Phi1}
\varphi_j=2ABC\frac{2(1-B)\binom{1-\frac2B}{j}-(1-2B)\binom{1-\frac1B}{j}}{(-C)^j},
\end{equation}
for all $j\in\N_0$, we obtain $-\frac{1}{B}\in\N$. Thus, the trees can be generated as weighted $b$-ary trees with $b=1-\frac{2}{B}\in\N$, 
such that only out-degrees from zero to $b$ are allowed. A concrete example would be the choice $A=1$, $B=-1$ and $C=-1$, 
such that $T(z)=\frac{1}{1-z^2}-1$, and $\varphi_j=8\binom{3}{j}-6\binom{2}{j}$, $j\in\N_0$. The positive degree weights are $\varphi_0=2$, $\varphi_1=12$, 
$\varphi_2=18$, $\varphi_3=8$ and $\varphi_j=0$, for $j\ge 4$.

\smallskip

In the case (ii) $0<B<1$, $A>0$ and $C>0$ we note first that the special case $B=\frac12$ leading to 
$$
T(z)=C\cdot(1-(1-Az^2)^{\frac12}),\qquad \varphi(T)=\frac{AC}{(1-\frac{T}{C}\big)^{3}}.
$$
The choice $A=C=1$ leads to the family of three-bundled bilabelled increasing trees discussed earlier. 
More generally, we have 
\[
\varphi(T)=2ABC\sum_{j\ge 0}\Big[(2-2B)\binom{2B-1+j}{j}-(1-2B)\binom{B-1+j}{j}\Big] C^j T^j.
\]
Consequently, the coefficients 
\begin{equation}
\label{Phi2}
\varphi_j=2ABC\sum_{j \ge 0}\Big[(2-2B)\binom{2B-1+j}{j}-(1-2B)\binom{B-1+j}{j}\Big] C^j,
\end{equation}
$j \in \N_{0}$, are non-negative, for all $0<B<1$, $C>0$ and $A>0$.

\smallskip 

Concerning hook-length formul\ae{} we interpret in both cases the weight sequences in terms of weighted ordered trees. 
For both families with degree-weights $\varphi_j$ given by~\eqref{Phi1} or~\eqref{Phi2} we get 
\[
\sum_{T\in\mathcal{O}(n)} \prod_{v\in T} \left(\frac{\varphi_{\odeg(v)}}{2h_{v}(2h_{v}-1)}\right) 
= (-C)\cdot (-A)^n\binom{B}{n}.
\]
\end{example}

\section{k-labelled increasing trees and hook-length formulas}
\subsection{Combinatorial description}
Similar to bilabelled increasing tree families we can consider in general $k$-labelled increasing tree families $\widehat{\mathcal{T}} = \widehat{\mathcal{T}}_{k}$, with $k\ge 1$. 
We call a tree $T$ a $k$-\emph{labelled} tree, if each node $v \in T$ has got a set $\ell_{k}(v)= \{\ell^{[1]}(v),\dots,\ell^{[k]}(v)\}$ of $k$-different integers.
We may always assume that $\ell^{[1]}(v) < \dots<\ell^{[k]}(v)$, and where furthermore the label sets of different nodes are disjoint, i.e.,
$\ell_{k}(v) \cap \ell_{k}(w) = \emptyset$, for $v \neq w$. We say then that $T$ is a $k$-labelled tree with \emph{label set} 
$\mathcal{M} = \mathcal{M}(T) = \bigcup_{v \in T} \ell_{k}(v)$; of course, $|\mathcal{M}| = kn$, for a tree $T$ of size $|T|=n$. 

A $k$-labelled tree $T$ is called \emph{increasing}, if it holds that each label of a child node is always larger than all labels of its parent node:
$\ell_{k}(v) \prec \ell_{k}(w)$, whenever $w$ is a child of $v$, where we use the relation $\{a^{[1]},\dots,a^{[k]}\} \prec \{b^{[1]},\dots,b^{[k]}\} \Longleftrightarrow \max_{i} a^{[i]} < \min_{j} b^{[j]}$.
In Figure~\ref{fig:klabelled} we give an example of a $3$-labelled increasing tree.
\begin{figure}
\begin{center}
  \includegraphics[height=3.5cm]{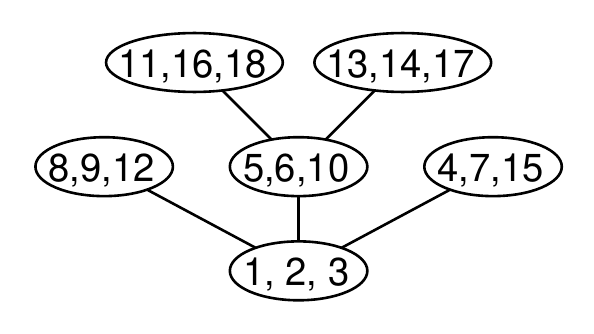}
\end{center}
\caption{A $3$-labelled increasing tree of size $6$, i.e., with $18$ labels.\label{fig:klabelled}}
\end{figure}

We denote by $\widehat{\mathcal{T}}=\widehat{\mathcal{T}}_{k}$ the family of \emph{increasingly $k$-labelled weighted ordered trees},
which contains all (non-empty) increasingly $k$-labelled ordered trees $T \in \mathcal{O}$ of size $|T| \ge 1$ with label set $\mathcal{M} = \{1, 2, \dots, k|T|\}$
and given degree-weight generating function $\varphi(t)$. Analogous to \eqref{HookBijII-eqn1}, such a family $\widehat{\mathcal{T}}$ can be described by the combinatorial construction 
\[
	\widehat{\mathcal{T}} = \big(\mathcal{Z}^{\Box}\big)^k \ast \varphi\big(\widehat{\mathcal{T}}\big).
\]
Equivalently, the exponential generating function $T(z)=\sum_{n\ge 1}T_n\frac{z^{kn}}{(kn)!}$ satisfies the differential equation (with $D_{z}$ denoting the differential operator w.r.t.\ $z$):
\begin{equation}\label{eqn:klabelDEQ}
D_z^k \left(T(z)\right) = T^{(k)}(z) = \varphi\big(T(z)\big),\qquad T^{(\ell)}(z)=0,\quad 0\le \ell \le k-1.
\end{equation}
One can readily adapt the reverse engineering approach presented in Section~\ref{sec:ReverseEngineering} to $k$-labelled increasing tree families. 

\subsection{Hook-length formulas}

Moreover, in order to derive hook-length formul\ae{} we can generalize Lemma~\ref{HookBijII-lem1}, which has been shown in \cite{KubPan2012}, for the number of increasing bilabellings of a given tree $T$ to increasing $k$-labellings.
\begin{lemma}
\label{HookBijOutlookLem1}
  The number $|\mathcal{L}^{[k]}(T)|$ of different increasing $k$-labellings of a tree $T$ of size $n$ with distinguishable nodes is given as follows:
  \begin{equation*}
    |\mathcal{L}^{[k]}(T)| = \frac{(kn)!}{\prod_{v \in T}\left(\fallfak{(k\cdot h_{v})}{k}\right)}.
  \end{equation*}
\end{lemma}
\begin{proof}
  The formula can be shown easily by using induction on the size $|T|=n$ of $T$.
  For $n=1$ there is exactly one increasing $k$-labelling of $T$ and $h(v)=1$ such that $\frac{k!}{\fallfak{k}{k}}=1$.
  Let $n > 1$: we assume that the root of $T$ has \emph{out-degree} $r$. Let us denote the subtrees of the root,
    which have corresponding sizes $s_{1}, \dots, s_{r}$, by $T_{1}, \dots, T_{r}$.
    It holds that after an order preserving relabelling each of the subtrees $T_{1}, \dots, T_{r}$ is itself an increasingly $k$-labelled tree.
    Taking into account that the root node of $T$ is labelled by $\{1, \dots, k\}$ and that the remaining nodes are distributed over the nodes of the  subtrees one obtains:
    \begin{equation*}
      |\mathcal{L}^{[k]}(T)| = \binom{kn-k}{k s_{1}, k s_{2}, \dots, k s_{r}} \cdot |\mathcal{L}^{[k]}(T_{1})| \cdot |\mathcal{L}^{[k]}(T_{2})| \cdots |\mathcal{L}^{[k]}(T_{r})|.
    \end{equation*}
    Using the induction hypothesis we further get:
    \begin{equation*}
      |\mathcal{L}^{[k]}(T)| = \frac{(kn-k)!}{\prod_{j=1}^{r}(k s_{j})!} \prod_{j=1}^{r} \frac{(k s_{j})!}{\prod_{v \in T_{j}}\left(\fallfak{(k\cdot h_{v})}{k}\right)}
      = \frac{(kn)!}{\prod_{v \in T}\left(\fallfak{(k\cdot h_{v})}{k}\right)},
    \end{equation*}
    which completes the proof.
\end{proof}

As a consequence we obtain a generalization of Theorem~\ref{HookBijIITheHook}.
\begin{theorem}
\label{kHookFormula}
Given a family $\widehat{\mathcal{T}}$ of increasingly $k$-labelled weighted ordered trees with degree-weight generating function $\varphi(t) = \sum_{j \ge 0} \varphi_{j} t^{j}$.
Then, the family $\mathcal{O}$ of ordered trees satisfies the following hook-length formula:
\[
\sum_{T\in\mathcal{O}(n)}\prod_{v\in T} \left(\frac{\varphi_{\odeg(v)}}{\fallfak{(k h_{v})}k}\right) = \frac{T_n}{(kn)!}.
\]
\end{theorem}

\subsection{Trilabelled unordered increasing trees}
We consider the family of unordered trilabelled increasing trees with 
degree-weight generating function $\varphi(t)=e^t$. The Blasius function $y(z)$ is the solution of the so-called Blasius differential equation 
\[
y'''(z)+y''(z)y(z)=0,\qquad y(0)=0,\quad y'(0)=0,\quad \lim_{z\to\infty}y'(z)=1,
\]
arising in the study of the Prandtl-Blasius Flow~\cite{Blasius1908,BlasiusBoyd1999,BlasiusFinch2008,BlasiusHager2003}.
It is known that the solution of this differential equation 
satisfies the power series expansion 
\[
y(z)=\sum_{n=0}^\infty (-1)^n \frac{p_n \xi^{n+1}}{(3n+2)!}z^{3n+2}, \quad \xi=y''(0)\doteq 0.4695999883,
\]
with $p_0=1$ and $p_n$ a certain positive integer sequence, see~\cite{Blasius1908,BlasiusFinch2008}. We obtain the following result.

\begin{prop}
The derivative $F(z)=T'(z)$ of the exponential generating function $T(z)$ of unordered trilabelled increasing trees
can be given in terms of the Blasius function $y(z)$ as follows:
\[
F(z)=\frac{1}{\xi^{\frac13}}y\Big(\frac{-z}{\xi^{\frac13}}\Big),
\] 
Moreover, the numbers $T_n$, counting unordered trilabelled increasing trees with $3n$ labels, are the shifted Blasius numbers:
$T_{n+1}=p_n$. They satisfy the recurrence relation
\[
T_{n+1}=
\sum_{k=1}^{n}\binom{3n-1}{3k-3}T_{k}T_{n-k},\quad n>1,\quad T_1=1.
\]
\end{prop}

\begin{remark}
The sequence $(T_n)$ starts with
\begin{equation*}
  (T_n)=(1, 1, 11, 375, 27897, 3817137,\dots),
\end{equation*}
compare with Figure~\ref{HookBijIIFigureUnordered}; in the OEIS they appear as \href{https://oeis.org/A018893}{A018893} without a combinatorial interpretation.
\end{remark}
\begin{proof}
According to Equation~\ref{eqn:klabelDEQ} the exponential generating function $T(z)$ satisfies the third order non-linear autonomous differential equation
\[
T'''(z)=e^{T(z)},\quad T(0)=0,\quad T'(0)=0,\quad T''(0)=0. 
\]
Let $F(z)=T'(z)$. We translate the equation $T''''(z)=e^{T(z)}T'(z)$ into a differential equation for $F(z)$:
\[
F'''(z) = F''(z)F(z),
\]
or equivalently
\[
F'''(z) - F''(z)F(z)=0, \quad F(0)=0,\quad F'(0)=0,\quad F''(0)=1. 
\]
We directly obtain the stated recurrence relation by extraction of coefficients from $F(z)=\sum_{n\ge 1}T_n\frac{z^{3n-1}}{(3n-1)!}$.
Moreover, we observe that the differential equation for $F(-z)$ is identical to the Blasius differential equation except the last initial condition. The relation of $F(z)=T'(z)$ to the Blasius function is now readily obtained by matching the power series expansion of $F(z)$
to the known expansion of $y(x)$. 
\end{proof}


Since $\varphi_{j}=[t^{j}]e^t = \frac{1}{j!}$, Theorem~\ref{kHookFormula} gives the following hook-length formula.
\begin{coroll}
The family $\mathcal{O}$ of ordered trees satisfies the following hook-length formula:
\[
  \sum_{T\in \mathcal{O}(n)} \prod_{v \in T} \left(\frac{1}{\deg(v)! \cdot 3h_{v}(3h_{v}-1)(3h_{v}-2)}\right) = \frac{p_{n-1}}{(3n)!},
\]
where $(p_n)_{n\ge 0}$ denote the coefficients of the Blasius function $y(z)$. 
\end{coroll}

\section{Free multilabelled increasing trees\label{sec:FreeMul}}

\subsection{Definition and enumeration results}
So far we considered increasing labellings, where each node in the tree has got the same number of labels. In order to get rid of this restriction, we change the point of few and think of $m$ labels, which we want to distribute over the nodes of trees of size $n \le m$ in an ``increasing way''. Given a label set $\mathcal{M}$ of $m = |\mathcal{M}|$ labels, we call a tree $T$ a \emph{free multilabelled} tree with label set $\mathcal{M}$, if each node $v \in T$ has got a non-empty set $\ell(v) \subset \mathcal{M}$ of labels satisfying $\ell(v) \cap \ell(w) = \emptyset$, for $v \neq w$, and $\bigcup_{v \in T} \ell(v) = \mathcal{M}$. In other words, the labellings of the nodes $v \in T$ are forming a partition of the label set $\mathcal{M}$. A free multilabelled tree $T$ is called increasing, if it holds that each label of a child node is always larger than all labels of its parent node: 
$\ell(v) \prec \ell(w)$, whenever $w$ is a child of $v$, where we use the relation $\{a^{[i]}: i \in I\} \prec \{b^{[j]}: j \in J\} \Longleftrightarrow \max_{i} a^{[i]} < \min_{j} b^{[j]}$.

We denote by $\widehat{\mathcal{T}}$ the family of increasingly free multilabelled weighted ordered trees (= free multilabelled increasing trees), which contains all increasingly free multilabelled trees $T \in \mathcal{O}$ of size $|T| \ge 1$ with label sets $\mathcal{M} = \mathcal{M}(T) = \{1, 2, \dots, |\mathcal{M}|\}$ and $|\mathcal{M}| \ge |T|$, and given degree-weight generating function $\varphi(t)$.
In Figure~\ref{fig:multilabelled} we give all ordered multilabelled increasing trees with $3$ labels.
\begin{figure}
\begin{center}
  \includegraphics[height=3.5cm]{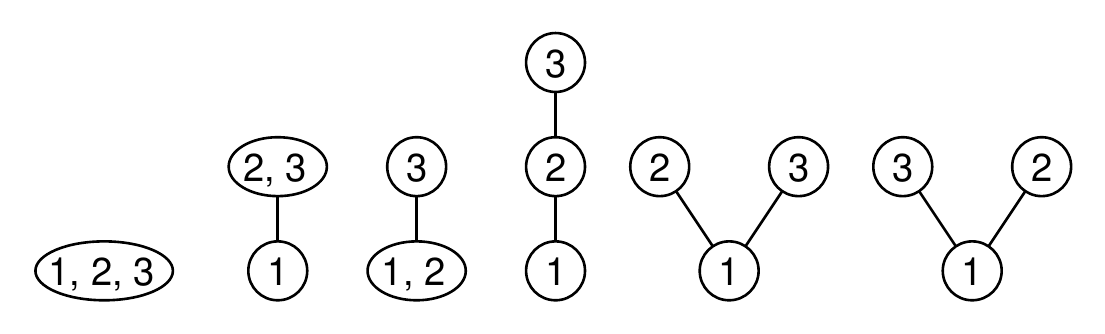}
\end{center}
\caption{All ordered multilabelled increasing trees with $3$ labels.\label{fig:multilabelled}}
\end{figure}

Let us denote by $T_{m}$ the number of free multilabelled increasing trees of $\widehat{\mathcal{T}}$ with $m$ labels, i.e., label set $\mathcal{M} = [m]$, and by $T(z) := \sum_{m \ge 1} T_{m} \frac{z^{m}}{m!}$ the exponential generating function. We get the following simple characterization of the generating function $T(z)$.
\begin{prop}\label{prop:FreeMulti}
The exponential generating function $T(z)$ of the number of free multilabelled increasing trees $T_{m}$ with $m$ labels and degree-weight generating function $\varphi(t)$ satisfies the following first order differential equation:
\begin{equation}\label{eqn:FreeMultiDEQ}
  T'(z) = \varphi(T(z)) + T(z), \quad T(0)=0.
\end{equation}
\end{prop}
\begin{proof}
  The number $T_{m}$ of free multilabelled increasing trees with $m$ labels can be counted recursively by distinguishing a tree into the root node and its $r \ge 0$ subtrees and taking into account the number $j \ge 1$ of labels the root node as well as the numbers $s_{1}, \dots, s_{r}$ of labels the subtrees will get. After an order preserving relabelling each of the subtrees is itself a free multilabelled increasing tree, which yields the following recurrence.
	\begin{equation*}
	  T_{m} = \sum_{j=1}^{m} \sum_{r \ge 0} \varphi_{r} \sum_{s_{1} + \cdots + s_{r} = m-j} \binom{m-j}{s_{1}, s_{2}, \dots, s_{r}} T_{s_{1}} T_{s_{2}} \cdots T_{s_{r}}, \quad m \ge 1, 
	\end{equation*}
	with $T_{0}=0$.	Considering the cases $j=1$ and $j > 1$ separately, we further obtain
	\begin{align*}
	  T_{m} & = \sum_{r \ge 0} \varphi_{r} \sum_{s_{1} + \cdots + s_{r} = m-1} \binom{m-1}{s_{1}, s_{2}, \dots, s_{r}} T_{s_{1}} T_{s_{2}} \cdots T_{s_{r}} \\
		& \quad \mbox{} + \sum_{j=2}^{m} \sum_{r \ge 0} \varphi_{r} \sum_{s_{1} + \cdots + s_{r} = m-j} \binom{m-j}{s_{1}, s_{2}, \dots, s_{r}} T_{s_{1}} T_{s_{2}} \cdots T_{s_{r}}\\
		& = \sum_{r \ge 0} \varphi_{r} \sum_{s_{1} + \cdots + s_{r} = m-1} \binom{m-1}{s_{1}, s_{2}, \dots, s_{r}} T_{s_{1}} T_{s_{2}} \cdots T_{s_{r}} \\
		& \quad \mbox{} + \sum_{j=1}^{m-1} \sum_{r \ge 0} \varphi_{r} \sum_{s_{1} + \cdots + s_{r} = m-1-j} \binom{m-1-j}{s_{1}, s_{2}, \dots, s_{r}} T_{s_{1}} T_{s_{2}} \cdots T_{s_{r}}\\
		& = \sum_{r \ge 0} \varphi_{r} \sum_{s_{1} + \cdots + s_{r} = m-1} \binom{m-1}{s_{1}, s_{2}, \dots, s_{r}} T_{s_{1}} T_{s_{2}} \cdots T_{s_{r}} + T_{m-1}, \quad m \ge 1,
	\end{align*}
	with $T_{0}=0$.	Treating this recurrence with the generating function $T(z) = \sum_{m \ge 1} T_{m} \frac{z^{m}}{m!}$, we immediately get the differential equation stated above.
\end{proof}

Comparing the differential equation \eqref{eqn:FreeMultiDEQ} for the generating function of the number of free multilabelled increasing trees with the well-known differential equation for the generating function of the number of (unilabelled) increasing trees, which is the instance $k=1$ in \eqref{eqn:klabelDEQ}, we immediately get from Proposition~\ref{prop:FreeMulti} the following corollary.
\begin{coroll}\label{cor:FreeMultiUniInc}
  Let $\widehat{\mathcal{T}}$ denote the family of free multilabelled increasing trees with degree-weight generating function $\varphi(t)$ and $(T_{m})$ the sequence of the number of such trees with $m$ labels. Moreover, let $\widetilde{\mathcal{T}}$ denote the family of (unilabelled) increasing trees with degree-weight generating function $\widetilde{\varphi}(t) := \varphi(t)+t$ and $(\widetilde{T}_{m})$ the sequence of the number of such trees with $m$ nodes (= labels). Then it holds for all $m \ge 1$:
	\begin{equation*}
	  T_{m} = \widetilde{T}_{m}.
	\end{equation*}
\end{coroll}

In the following we give a combinatorial interpretation of this link between free multilabelled increasing trees and unilabelled increasing trees. We formulate this relation for the family of ordered trees, but it is straightforward to extended it to weighted ordered trees.
\begin{theorem}
  There is a bijection from the family $\widehat{\mathcal{O}}(m)$ of ordered free multilabelled increasing trees with $m$ labels to the family $\widetilde{\mathcal{O}}(m)$ of ordered increasing trees of size $m$, where each node of out-degree $1$ is either coloured black or white.
\end{theorem}
\begin{proof}
  Consider an ordered free multilabelled increasing tree $T \in \widehat{\mathcal{O}}(m)$ with label set $[m]$.
	In order to construct an ordered increasing tree $\widetilde{T} \in \widetilde{\mathcal{O}}(m)$ of size $m$, where each node of out-degree $1$ is either coloured black or white, do the following procedure for each node $v \in T$. Let $\ell(v) = \{\ell^{[1]}(v), \dots, \ell^{[k]}(v)\}$, for a $k \ge 1$, be the label set associated to $v$, where we assume that $\ell^{[1]}(v) < \cdots < \ell^{[k]}(v)$. Then replace $v$ by a chain $v_{1} - v_{2} - \cdots - v_{k}$ of $k$ nodes with respective labels $\ell^{[1]}(v), \ell^{[2]}(v), \dots, \ell^{[k]}(v)$, where the first $k-1$ nodes, i.e., $v_{1}, \dots, v_{k-1}$, are coloured black and $v_{k}$ is coloured white. Moreover, the original subtrees of $v$ will be the subtrees of $v_{k}$.
	
	Obviously, the resulting tree $\widetilde{T}$ is contained in $\widetilde{\mathcal{O}}(m)$, and the original tree $T \in \widehat{\mathcal{O}}(m)$ can be reobtained in a straightforward way by pushing together chains of black nodes with the subsequent white node. The bijection is exemplified in Figure~\ref{fig:MultBij}.
\end{proof}
\begin{figure}
\begin{center}
  \includegraphics[height=4cm]{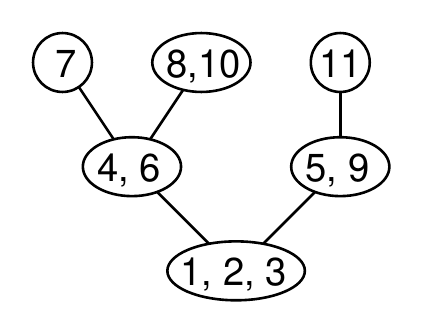} \raisebox{2.5cm}{$\Longleftrightarrow$} \includegraphics[height=5cm]{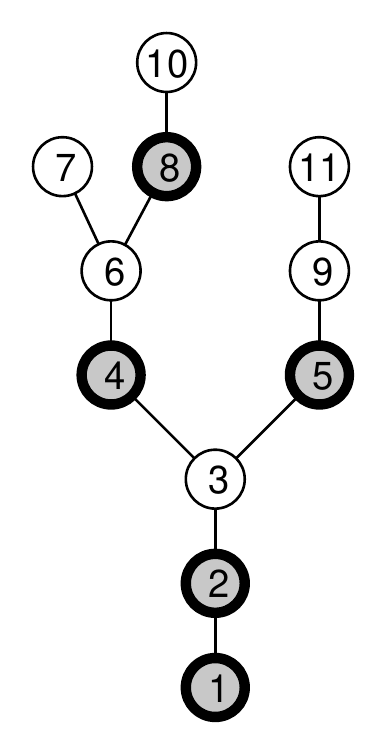}
\end{center}
\caption{An ordered free multilabelled increasing tree with $11$ labels and the corresponding ordered increasing tree of size $11$, where nodes of out-degree $1$ are coloured either black or white.\label{fig:MultBij}}
\end{figure}

\subsection{Combinatorial families of free multilabelled increasing trees}
Here we give a few concrete examples of combinatorial families of free multilabelled increasing trees. Of course, due to Corollary~\ref{cor:FreeMultiUniInc}, many well-known enumeration results for unilabelled increasing trees have its correspondence in results for free multilabelled increasing trees.

\begin{example}[\emph{Ordered trees without nodes of out-degree $1$}]
We consider the family $\widehat{\mathcal{T}}$ of free multilabelled increasing trees with degree-weight generating function $\varphi(t) = \sum_{j \ge 0} \varphi_{j} t^{j} = \frac{1}{1-t} - t$, i.e., $\varphi_{j}=1$, for $j \neq 1$, and $\varphi_{1}=0$. Due to Corollary~\ref{cor:FreeMultiUniInc} and the well-known enumeration result for unilabelled ordered increasing trees (often called plane-oriented recursive trees), see, e.g., \cite{BerFlaSal1992,MahSmy1995}, we obtain the following result.
\begin{coroll}
  The number $T_{m}$ of ordered free multilabelled increasing trees without nodes of out-degree $1$ and with $m$ labels is given as follows:
	\begin{equation*}
	  T_{m} = (2m-3)!! = \frac{(2m-2)!}{2^{m-1} \, (m-1)!}, \quad m \ge 1.
	\end{equation*}
\end{coroll}
\end{example}

\begin{example}[\emph{Unordered trees without nodes of out-degree $1$}]
We consider the family $\widehat{\mathcal{T}}$ of free multilabelled increasing trees with degree-weight generating function $\varphi(t) = \sum_{j \ge 0} \varphi_{j} t^{j} = e^{t} - t$, i.e., $\varphi_{j}=\frac{1}{j!}$, for $j \neq 1$, and $\varphi_{1}=0$. Again, Corollary~\ref{cor:FreeMultiUniInc} and the well-known enumeration result for unilabelled unordered increasing trees (usually called recursive trees), see, e.g., \cite{BerFlaSal1992,MahSmy1995}, yields the following result.
\begin{coroll}
  The number $T_{m}$ of unordered free multilabelled increasing trees without nodes of out-degree $1$ and with $m$ labels is given as follows:
	\begin{equation*}
	  T_{m} = (m-1)!, \quad m \ge 1.
	\end{equation*}
\end{coroll}
\end{example}

\begin{example}[\emph{Strict-binary trees}]
We consider the family $\widehat{\mathcal{T}}$ of strict-binary free multilabelled increasing trees, which correspond to increasingly free multilabelled weighted ordered trees with degree-weight generating function $\varphi(t) = \sum_{j \ge 0} \varphi_{j} t^{j} \newline = 1+t^{2}$. Corollary~\ref{cor:FreeMultiUniInc} shows that the number $T_{m}$ of such trees with $m$ labels is equal to the number of increasingly labelled unary-binary trees (= increasingly labelled Motzkin tree) of size $m$. Due to Proposition~\ref{prop:FreeMulti} the exponential generating function $T(z)$ of the number of trees with $m$ labels satisfies the differential equation
\begin{equation*}
  T'(z) = 1 + T(z) + T(z)^{2}, \quad T(0) = 0.
\end{equation*}
Solving this differential equation by standard methods and extracting coefficients yields the following results.
\begin{coroll}
  The exponential generating function $T(z) = \sum_{m \ge 1} T_{m} \frac{z^{m}}{m!}$ of the number $T_{m}$ of strict-binary free multilabelled increasing trees with $m$ labels is given as follows:
	\begin{equation*}
	  T(z) = \frac{1}{\sqrt{3} \cot\big(\frac{\sqrt{3} \, z}{2}\big) -1} = \frac{2 \left(e^{i \sqrt{3} \, z} - 1\right)}{\sqrt{3} i + 1 + (\sqrt{3} i -1) e^{i \sqrt{3} \, z}}.
	\end{equation*}
	The numbers $T_{m}$ are for $m \ge 1$ given by the following explicit formula:
	\begin{equation*}
	  T_{m} = \sum_{k=1}^{m} 3^{\frac{m-k}{2}} \cos\left(\frac{(3m+2-5k) \pi}{6}\right) \sum_{\ell=1}^{k} \binom{k}{\ell} (-1)^{k-\ell} \ell^{m}.
	\end{equation*}
\end{coroll}
\begin{remark}
The sequence $(T_m)$ begins with
\begin{equation*}
  (T_m)_{m \ge 1} = (1, 1, 3, 9, 39, 189, 1107,\dots)
\end{equation*}
and in the OEIS the numbers $T_{m}$ appear as \href{https://oeis.org/A080635}{A080635}.
\end{remark}
\end{example}

\begin{example}[\emph{Unary-binary trees}]
We consider the family $\widehat{\mathcal{T}}$ of unary-binary free multilabelled increasing trees (i.e., increasingly free multilabelled Motzkin trees), which correspond to increasingly free multilabelled weighted ordered trees with degree-weight generating function $\varphi(t) = \sum_{j \ge 0} \varphi_{j} t^{j} = 1+t+t^{2}$. According to Corollary~\ref{cor:FreeMultiUniInc}, the number $T_{m}$ of such trees with $m$ labels is enumerated by the number of unilabelled binary increasing trees of size $m$. The well-known enumeration result for the latter tree family (see, e.g., \cite{BerFlaSal1992}) immediately leads to the following result.
\begin{coroll}
  The number $T_{m}$ of unary-binary free multilabelled increasing trees with $m$ labels is given as follows:
	\begin{equation*}
	  T_{m} = m!, \quad m \ge 1.
	\end{equation*}
\end{coroll}
\end{example}

\begin{example}[\emph{Binary trees}]
We consider the family $\widehat{\mathcal{T}}$ of binary free multilabelled increasing trees, which correspond to increasingly free multilabelled weighted ordered trees with degree-weight generating function $\varphi(t) = \sum_{j \ge 0} \varphi_{j} t^{j} = (1+t)^{2}$.
Due to Proposition~\ref{prop:FreeMulti} the exponential generating function $T(z)$ of the number of trees with $m$ labels satisfies the differential equation
\begin{equation*}
  T'(z) = 1 + 3 T(z) + T(z)^{2}, \quad T(0) = 0.
\end{equation*}
Solving this differential equation by standard methods and extracting coefficients yields the following results.
\begin{coroll}
  The exponential generating function $T(z)$ of the number $T_{m}$ of binary free multilabelled increasing trees with $m$ labels is given as follows:
	\begin{equation*}
	  T(z) = \frac{2 \left(1-e^{\sqrt{5} \, z}\right)}{(3-\sqrt{5}) e^{\sqrt{5} \, z} - 3 - \sqrt{5}}.
	\end{equation*}
	The numbers $T_{m}$ are for $m \ge 1$ given by the following explicit formula:
	\begin{equation*}
	  T_{m} = \sqrt{5} \sum_{k \ge 1} \left(\frac{7-3\sqrt{5}}{2}\right)^{k} \cdot \left(\sqrt{5} \, k\right)^{m}.
	\end{equation*}
\end{coroll}
\begin{remark}
The sequence $(T_m)$ begins with
\begin{equation*}
  (T_m)_{m \ge 1} = (1, 3, 11, 51, 295, 2055, 16715,\dots)
\end{equation*}
and in the OEIS the numbers $T_{m}$ appear as \href{https://oeis.org/A230008}{A230008}, but without giving a combinatorial interpretation of this enumeration sequence.
\end{remark}
\end{example}

\subsection{Free multilabelled increasing trees and hook-length formul\ae{}\label{ssec:FreeMultHook}}
In order to formulate hook-length formul\ae{} associated to free multilabelled increasing trees we consider a tree $T$ and assume that each node $v \in T$ has got a certain bucket-size $b(v) \in \mathbb{N}_{\ge 1}$, i.e., node $v$ can hold $b(v)$ labels. Let us assume that the total bucket-size of $T$ is $m = \sum_{v \in T} b(v)$, such that $T$ can hold $m$ labels in total. We might consider then $b : T \to \mathbb{N}_{\ge 1}$ as a bucket-size function with $m$ labels for $T$. In this context it is useful to consider an extension of the term hook-length of a node $v \in T$, which we call bucket hook-length $h^{[b]}(v)$ of $v$, which is defined as the sum of the bucket-sizes of all descendants of $v$:
\begin{equation*}
  h^{[b]}(v) := \sum_{\begin{smallmatrix} u \in T:\\ \text{$u$ descendant of $v$}\end{smallmatrix}} b(u).
\end{equation*}
We can then formulate a generalization of Lemma~\ref{HookBijOutlookLem1} for the number of increasing multilabellings of $T$ taking into account the bucket-sizes of the nodes of $T$.
\begin{lemma}
\label{HookBijMulti}
  Given a bucket-size function $b : T \to \mathbb{N}_{\ge 1}$ with $m$ labels for a tree $T$ with distinguishable nodes.
  Then it holds that the number $|\mathcal{L}(T)|$ of different increasing multilabellings of $T$ with $m$ labels, such that each node $v \in T$ gets exactly $b(v)$ labels, is given as follows:
  \begin{equation*}
    |\mathcal{L}(T)| = \frac{m!}{\prod_{v \in T}\left(h^{[b]}(v)\right)^{\underline{b(v)}}}.
  \end{equation*}
\end{lemma}
\begin{proof}
  The formula can be shown easily by using induction on the number $m$ of labels (i.e., the total bucket-size of $T$) distributed amongst $T$.
  For $m=1$ there is exactly one increasing multilabelling of $T$ and $h^{[b]}(v)=1$ such that the formula holds.
  Let $m > 1$: we assume that the root of $T$ has \emph{out-degree} $r$ and that its bucket-size is $b(\text{root})$. Let us denote the subtrees of the root by by $T_{1}, \dots, T_{r}$, and let us assume that the total bucket-sizes of the subtrees are given by $s_{j} := \sum_{v \in T_{j}} b(v)$, for $1 \le j \le r$.
    It holds that after an order preserving relabelling each of the subtrees $T_{1}, \dots, T_{r}$ is itself a multilabelled tree.
    Taking into account that the root node of $T$ is labelled by $\{1, \dots, b(\text{root})\}$ and that the remaining nodes are distributed over the nodes of the subtrees one obtains:
    \begin{equation*}
      |\mathcal{L}(T)| = \binom{m-b(\text{root})}{s_{1}, s_{2}, \dots, s_{r}} \cdot |\mathcal{L}(T_{1})| \cdot |\mathcal{L}(T_{2})| \cdots |\mathcal{L}(T_{r})|.
    \end{equation*}
    Using the induction hypothesis we further get:
    \begin{equation*}
      |\mathcal{L}(T)| = \frac{(m-b(\text{root}))!}{\prod_{j=1}^{r}s_{j}!} \prod_{j=1}^{r} \frac{s_{j}!}{\prod_{v \in T_{j}}\left(h^{[b]}(v)\right)^{\underline{b(v)}}}
      = \frac{m!}{\prod_{v \in T}\left(h^{[b]}(v)\right)^{\underline{b(v)}}},
    \end{equation*}
    which completes the proof.
\end{proof}

By considering all bucket-size functions with $m$ labels, we get the following hook-length formula for ordered trees as an immediate consequence of Lemma~\ref{HookBijMulti}.
\begin{prop}
Given a family $\widehat{\mathcal{T}}$ of increasingly free multilabelled weighted ordered trees with degree-weight generating function $\varphi(t) = \sum_{j \ge 0} \varphi_{j} t^{j}$, let us denote by $T_{m}$ the number of trees of $\widehat{\mathcal{T}}$ with $m$ labels.
Then, the family $\mathcal{O}$ of ordered trees satisfies the following hook-length formula:
\begin{equation*}
\sum_{T\in\mathcal{O}} \sum_{\begin{smallmatrix} b \: : \: T \to \mathbb{N}_{\ge 1},\\ \text{with} \; \sum_{v \in T} b(v) = m\end{smallmatrix}} \prod_{v\in T} \left(\frac{\varphi_{\odeg(v)}}{\left(h^{[b]}(v)\right)^{\underline{b(v)}}}\right) = \frac{T_{m}}{m!}.
\end{equation*}
\end{prop}

\section{Unilabelled-bilabelled increasing trees}

\subsection{Definition and enumeration results}
According to the definition of families of free multilabelled increasing trees as introduced in Section~\ref{sec:FreeMul}, such tree families do not have any restriction on the number of labels a node can get. However, we can request that each node only has a capacity $k$ and thus can only hold up to $k$ labels. We will consider here exclusively the case $k=2$, which means that each node in the tree gets either one or two labels. More formally, given a label set $\mathcal{M}$ of $m = |\mathcal{M}|$ labels, we call a tree $T$ a \emph{unilabelled-bilabelled} tree with label set $\mathcal{M}$, if each node $v \in T$ has got a set $\ell(v) \subset \mathcal{M}$ of labels of size $1 \le |\ell(v)| \le 2$ satisfying $\ell(v) \cap \ell(w) = \emptyset$, for $v \neq w$, and $\bigcup_{v \in T} \ell(v) = \mathcal{M}$. A unilabelled-bilabelled tree $T$ is called increasing, if it holds that each label of a child node is always larger than all labels of its parent node.

We denote by $\widehat{\mathcal{T}}$ the family of increasingly unilabelled-bilabelled weighted ordered trees (= unilabelled-bilabelled increasing trees), which contains all increasingly unilabelled-bilabelled trees $T \in \mathcal{O}$ of size $|T| \ge 1$ with label sets $\mathcal{M} = \mathcal{M}(T) = \{1, 2, \dots, |\mathcal{M}|\}$ and $|T| \le |\mathcal{M}| \le 2|T|$, and given degree-weight generating function $\varphi(t)$.

Let us denote by $T_{m}$ the number of unilabelled-bilabelled increasing trees of $\widehat{\mathcal{T}}$ with $m$ labels, i.e., label set $\mathcal{M} = [m]$, and by $T(z) := \sum_{m \ge 1} T_{m} \frac{z^{m}}{m!}$ the exponential generating function. Then $T(z)$ is characterized via the following differential equation.
\begin{prop}\label{prop:UniBi}
The exponential generating function $T(z)$ of unilabelled-bilabelled increasing trees $T_{m}$ with $m$ labels and degree-weight generating function $\varphi(t) = \sum_{j \ge 0} \varphi_{j} t^{j}$ satisfies the following second order differential equation:
\begin{equation}\label{eqn:UniBiDEQ}
  T''(z) = \varphi(T(z)) + T'(z) \varphi'(T(z)), \quad T(0)=0, \quad T'(0)=\varphi_{0}.
\end{equation}
\end{prop}
\begin{proof}
  The number $T_{m}$ of unilabelled-bilabelled increasing trees with $m$ labels can be counted recursively by distinguishing a tree into the root node and its $r \ge 0$ subtrees and taking into account, whether the root node gets one or two labels. Furthermore, we take into consideration the numbers $s_{1}, \dots, s_{r}$ of labels the subtrees of the root will get. Since, after an order preserving relabelling, the subtrees of the root are itself unilabelled-bilabelled increasing trees, we obtain the following recurrence.
	\begin{align*}
	  T_{m} & = \sum_{r \ge 0} \varphi_{r} \sum_{s_{1} + \cdots + s_{r} = m-1} \binom{m-1}{s_{1}, s_{2}, \dots, s_{r}} T_{s_{1}} T_{s_{2}} \cdots T_{s_{r}}\\
		& \quad \mbox{} + \sum_{r \ge 0} \varphi_{r} \sum_{s_{1} + \cdots + s_{r} = m-2} \binom{m-2}{s_{1}, s_{2}, \dots, s_{r}} T_{s_{1}} T_{s_{2}} \cdots T_{s_{r}}, \quad m \ge 2,
	\end{align*}
	with $T_{0}=0$ and $T_{1} = \varphi_{0}$.
	After standard computations, this recurrence leads to the stated differential equation for the generating function $T(z) = \sum_{m \ge 1} T_{m} \frac{z^{m}}{m!}$.
\end{proof}

\subsection{Unordered unilabelled-bilabelled increasing trees}
We study the combinatorial family $\widehat{\mathcal{T}}$ of unordered unilabelled-bilabelled increasing trees, i.e., unilabelled-bilabelled increasing trees with degree-weight generating function $\varphi(t) = e^{t}$.
\begin{theorem}\label{thm:UnorderedUniBi}
Let $Q(z) = \sum_{m \ge 1} Q_{m} \frac{z^{m}}{m!}$ be defined implicitly via
\begin{equation*}
  z = \int_{0}^{Q(z)} \frac{dt}{2 e^{t} - t -1}.
\end{equation*}
Then, the derivative of the exponential generating function $T(z) = \sum_{m \ge 1} T_{m} \frac{z^{m}}{m!}$ of the number $T_{m}$ of unordered unilabelled-bilabelled increasing trees with $m$ labels can be expressed in terms of $Q(z)$ as follows:
\begin{equation*}
  T'(z) = Q(z) + Q'(z).
\end{equation*}
Thus, the numbers $T_{m}$ are given by
\begin{equation*}
  T_{m} = Q_{m} + Q_{m-1},
\end{equation*}
where the numbers $Q_{m}$ satisfy the following recurrence:
\begin{equation*}
  Q_{m+2} = \sum_{k=0}^{m} \binom{m}{k} \left(Q_{k} + Q_{k+1}\right) Q_{m-k+1}, \quad m \ge 0, \quad Q_{0}=0, \enspace Q_{1}=1.
\end{equation*}
\end{theorem}
\begin{proof}
  According to Proposition~\ref{prop:UniBi}, the generating function $T(z)$ satisfies a second order differential equation, which can be rewritten as follows:
	\begin{equation}\label{eqn:DEQ2UniBi}
	  (T'(z)-\varphi(T(z)))' = \varphi(T(z)).
	\end{equation}
	Setting
	\begin{equation*}
	  Q(z) := T'(z) - \varphi(T(z)),
	\end{equation*}
	equation~\eqref{eqn:DEQ2UniBi} gives 
	\begin{equation}\label{eqn:QzTz}
	  Q'(z) = \varphi(T(z)),
	\end{equation}
	which shows the stated relation between $T(z)$ and $Q(z)$:
	\begin{equation}\label{eqn:TzQz}
	  T'(z) = Q(z) + Q'(z).
	\end{equation}
	Differentiating \eqref{eqn:QzTz} and using $\varphi(t) = e^{t}$ yields the following differential equation for $Q(z)$, with initial conditions $Q(0) = 0$ and $Q'(0) = 1$:
	\begin{equation}\label{eqn:DEQ2Qz}
	  Q''(z) = \varphi'(T(z)) (Q(z) + Q'(z)) = Q'(z) \left(Q(z) + Q'(z)\right).
	\end{equation}
	This second order autonomous differential equation can be treated in a standard way by introducing the function $f(Q) := Q'(z)$, which gives the following first order linear differential equation for $f(Q)$:
	\begin{equation*}
	  f'(Q) = f(Q) + Q.
	\end{equation*}
	This differential equation has the following general solution:
	\begin{equation*}
	  f(Q) = c e^{Q} -Q - 1,
	\end{equation*}
	with an arbitrary constant $c$. Thus, we get
	\begin{equation*}
	  Q'(z) = c e^{Q(z)} - Q(z) - 1,
	\end{equation*}
	and adapting to the initial conditions $Q(0)=0$ and $Q'(0)=1$ characterizes the constant as $c=2$. Therefore, $Q(z)$ satisfies the following first order differential equation:
	\begin{equation}\label{eqn:DEQ1Qz}
	  Q'(z) = 2 e^{Q(z)} - Q(z) -1, \quad Q(0)=0.
	\end{equation}
	Solving \eqref{eqn:DEQ1Qz} by separating variables and integrating leads to the implicit characterization of $Q(z)$ stated in the theorem.
	
	\smallskip
	
	Extracting coefficients from relation \eqref{eqn:TzQz} immediately yields the connection between the numbers $T_{m}$ and $Q_{m}$; furthermore, the recurrence relation for $Q_{m}$ follows by extracting coefficients from differential equation~\eqref{eqn:DEQ2Qz}.
\end{proof}
\begin{remark}
The sequences $(Q_{m})$ and $(T_{m})$ occurring in above theorem begin with
\begin{align*}
  (Q_{m})_{m \ge 1} & = (1, 1, 3, 11, 55, 337, 2469,\dots),\\
  (T_{m})_{m \ge 1} & = (1, 2, 4, 14, 66, 392, 2806,\dots);
\end{align*}
currently, both sequences do not appear in the OEIS.
\end{remark}

According to differential equation \eqref{eqn:DEQ1Qz} for $Q(z)$, the numbers $Q_{m}$ count trees of size $m$ of a certain family $\widehat{\mathcal{Q}}$ of unilabelled increasing trees, i.e., of the one associated to the degree-weight generating function $\varphi(t) = 2 e^{t} - t -1$.
Combinatorially, such trees can be interpreted as unordered increasing trees, where each node of out-degree $\ge 2$ could be coloured either black or white (whereas nodes of out-degree zero or one are always coloured white). Next we give a combinatorial proof of the relation $T_{m} = Q_{m} + Q_{m-1}$ stated in Theorem~\ref{thm:UnorderedUniBi}.
\begin{theorem}\label{thm:TQbij}
  There is a bijection between the set $\widehat{\mathcal{T}}(m)$ of unordered unilabelled-bilabelled increasing trees with $m$ labels and the set $\widehat{\mathcal{Q}}(m) \: \dot{\cup} \: \widehat{\mathcal{Q}}(m-1)$ of unordered increasing trees of size $m$ or $m-1$, where each node of out-degree $\ge 2$ could be coloured either black or white.
\end{theorem}
\begin{proof}
  Consider a tree $T \in \widehat{\mathcal{T}}(m)$; if the root of $T$ has one label (i.e., it is labelled by $1$), then the recursive procedure below will map $T$ to a tree $Q \in \widehat{\mathcal{Q}}(m)$, but if the root of $T$ has two labels (i.e., it is labelled by $\{1, 2\}$), $T$ will be mapped to a tree $Q \in \widehat{\mathcal{Q}}(m-1)$. In the latter case we first remove label $1$ from the root of $T$ (thus keeping only label $2$), yielding a tree $T'$, and then carry out this recursive procedure.
	
	In order to describe the mapping we always assume that the children of any node in the tree are ordered from left to right in that way, such that the smaller labels in the respective nodes are forming an ascending sequence. Now carry out the procedure below starting by examining the root node $v$ of $T$ or $T'$, respectively.
	
	\begin{itemize}
	\item If $v$ has out-degree $0$ then colour $v$ white and return.
	
	\item Otherwise, let $v_{1}, \dots, v_{r}$ be the children of $v$ and $T_{1}, \dots, T_{r}$ be the respective subtrees.
	\begin{itemize}
	\item[$\ast$] If all nodes $v_{1}, \dots, v_{r}$ have only one label then do the following.
	\begin{enumerate}
	  \item Colour $v$ white.
	  \item Carry out this procedure to the nodes $v_{1}, \dots, v_{r}$.
	\end{enumerate}
	
	\item[$\ast$] Otherwise, let $v_{p}$ be the first node (from left to right) having two labels, let us say $\ell^{[1]}$ and $\ell^{[2]}$, whereas $v_{1}, \dots, v_{p-1}$ have only one label. Let us denote by $T_{p}^{[1]}, \dots, T_{p}^{[j]}$ the (possibly empty) subtrees of the node $v_{p}$.
	Do the following.
	\begin{enumerate}
	  \item Split node $v_{p}$, i.e., replace $v_{p}$ by two new nodes $v_{p}^{[1]}$ and $v_{p}^{[2]}$, which are children of $v$; $v_{p}^{[1]}$ and $v_{p}^{[2]}$ will get the labels $\ell^{[1]}$ and $\ell^{[2]}$, respectively. 
		\item Attach to node $v_{p}^{[2]}$ all the subtrees $T_{p}^{[1]}, \dots, T_{p}^{[j]}$ of the original node $v_{p}$, whereas attach to node $v_{p}^{[1]}$ the remaining subtrees $T_{p+1}, \dots, T_{r}$ of $v$.
		\item Colour $v$ black.
		\item	Carry out this procedure to the nodes $v_{1}, \dots, v_{p-1}, v_{p}^{[1]}, v_{p}^{[2]}$.
	\end{enumerate}
	\end{itemize}
	\end{itemize}
	
	It is immediate to see that the resulting tree $Q$ is indeed a member of $\widehat{\mathcal{Q}}(m)$ or $\widehat{\mathcal{Q}}(m-1)$, respectively. Moreover, each black node $v$ in $Q$ means that the two rightmost children of $v$ are obtained by splitting them and above procedure can be inverted easily, yielding the original tree $T$ with $m$ labels or a tree $T'$ with $m-1$ labels; in the latter case the original tree $T$ is obtained from $T'$ by adding label $1$ to the root node. In Figure~\ref{fig:TQbij} above procedure is exemplified.
\end{proof}
\begin{figure}
\begin{center}
  \includegraphics[height=3.5cm]{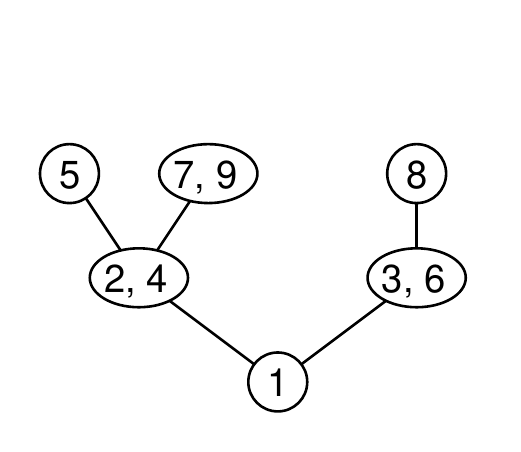} \raisebox{1.8cm}{$\Rightarrow$} \includegraphics[height=3.5cm]{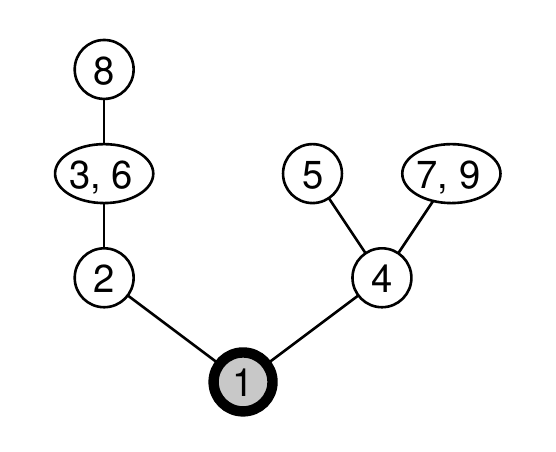} \raisebox{1.8cm}{$\Rightarrow$}	\includegraphics[height=3.5cm]{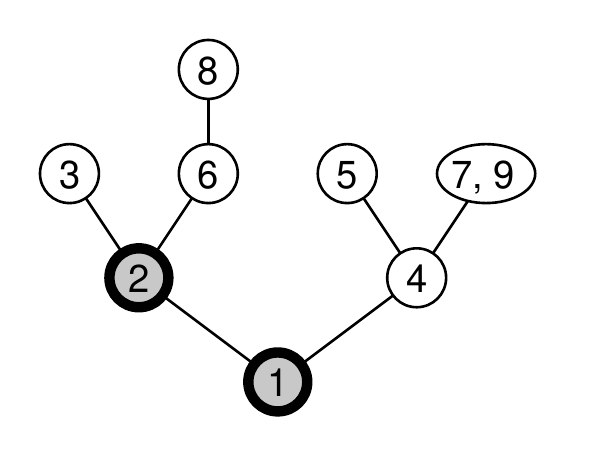} \raisebox{1.8cm}{$\Rightarrow$} \includegraphics[height=3.5cm]{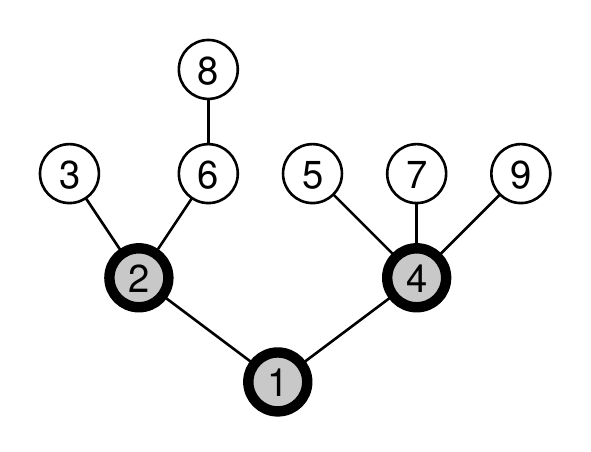}
\end{center}
\caption{An unordered unilabelled-bilabelled increasing tree $T$ with $9$ labels and the corresponding unordered increasing tree $Q$ of size $9$, where nodes of out-degree $\ge 2$ are coloured either black or white, obtained by the procedure given in Theorem~\ref{thm:TQbij}.\label{fig:TQbij}}
\end{figure}

\subsection{Unilabelled-bilabelled increasing trees and hook-length formul\ae{}}
The considerations made in Section~\ref{ssec:FreeMultHook} yielding relations between free multilabelled increasing trees and hook-length formul\ae{} can be carried over to unilabelled-bilabelled increasing trees easily, where one just has to take into account that the bucket-size $b(v)$ of any node $v$ in a tree can be only one or two. In particular, we get the following result.
\begin{prop}
Given a family $\widehat{\mathcal{T}}$ of increasingly unilabelled-bilabelled weighted ordered trees with degree-weight generating function $\varphi(t) = \sum_{j \ge 0} \varphi_{j} t^{j}$, let us denote by $T_{m}$ the number of trees of $\widehat{\mathcal{T}}$ with $m$ labels.
Then, the family $\mathcal{O}$ of ordered trees satisfies the following hook-length formula:
\begin{equation*}
\sum_{T\in\mathcal{O}} \sum_{\begin{smallmatrix} b \: : \: T \to \{1,2\},\\ \text{with} \; \sum_{v \in T} b(v) = m\end{smallmatrix}} \prod_{v\in T} \left(\frac{\varphi_{\odeg(v)}}{\left(h^{[b]}(v)\right)^{\underline{b(v)}}}\right) = \frac{T_{m}}{m!}.
\end{equation*}
\end{prop}

\section{\texorpdfstring{$k$}{k}-tuple labelled increasing trees}

We consider shortly another concept of increasing multilabellings of trees, where the nodes in the tree get $k$-tuples of labels, such that the $j$-th component of a child node is always larger than the $j$-th component of its parent node, for all $1 \le j \le k$.
Alternatively, we can interpret each $k$-tuple labelled increasing tree as a tree, to which a sequence of $k$ increasing (uni)labellings is associated. The particular instance $k=2$, called \emph{double increasing trees}, has been introduced by the authors in \cite{KubPan2012}, again in the context of combinatorial interpretations of hook-length formul\ae{}. Here we present for $k$-tuple labelled increasing tree families differential equations for a suitable generating function of the number of trees of size $n$ as well as relations to hook-length formul\ae{} for ordered trees. We further note, that it is possible to derive concrete hook-length formulas using the reverse-engineering approach presented in Section~\ref{sec:ReverseEngineering}.

\smallskip

We call a tree $T$ a \emph{$k$-tuple labelled} tree, if each node $v \in T$ has got an ordered $k$-tuple $\ell_{D}(v)= (\ell^{[1]}(v),\dots,\ell^{[k]}(v))$ of integers (we may speak about the \ith{j} label 
of $v$) such that the \ith{j} labels of two different nodes $v \neq w$ are different, $1\le j \le k$.
We say that $T$ is a $k$-tuple labelled tree with label sets $\mathcal{M}_{j} = \mathcal{M}_{j}(T) = \bigcup_{v \in T} \ell^{[j]}(v)$, $1\le j\le k$, respectively; of course, 
$|\mathcal{M}_{j}| = n$, for a tree $T$ of size $|T|=n$. A $k$-tuple labelled tree $T$ is called \emph{increasing}, 
if it holds that the \ith{j} label of a child node is always larger than the \ith{j} label of its parent node:
$\ell_{D}(v) \prec \ell_{D}(w)$, whenever $w$ is a child of $v$, where we use the relation $(a^{[1]},\dots,a^{[k]}) \prec (b^{[1]},\dots,b^{[k]}) \Longleftrightarrow \big(a^{[j]} < b^{[j]}$, $1\le j \le k$\big). In Figure~\ref{fig:ktuplelabelled} we give an example of a $3$-tuple labelled increasing tree.
\begin{figure}
\begin{center}
  \includegraphics[height=3.5cm]{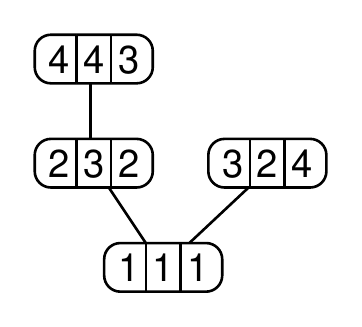}
\end{center}
\caption{A $3$-tuple labelled increasing tree of size $4$.\label{fig:ktuplelabelled}}
\end{figure}

\smallskip 

We denote by $\widehat{\mathcal{T}}_{k}$ the family of \emph{increasingly $k$-tuple labelled weighted ordered trees},
which contains all (non-empty) $k$-tuple increasingly labelled ordered trees $T \in \mathcal{O}$ of size $|T| \ge 1$ with label sets $\mathcal{M}_{i} = \{1, 2, \dots, |T|\}$ and degree-weight generating function $\varphi(t)$.
First, we derive a differential equation for the generating function $T(z)=\sum_{n\ge 1}T_n \frac{z^n}{(n!)^k}$.

\begin{prop}
The generating function $T(z)$ of the number $T_{n}$ of $k$-tuple labelled weighted ordered trees of size $n$ with degree-weight generating function $\varphi(t)$ satisfies the following differential equation:
 \[
 \frac1z \Theta_z^k \left(T(z)\right) = \varphi\left(T(z)\right),\qquad T^{\ell}(0)=\frac{T_\ell}{\ell!^k},\quad 0\le \ell \le k-1,
 \]
with differential operator $\Theta_z := z D_{z}$.
\end{prop}
\begin{proof}
We use the decomposition of a tree $T \in \widehat{\mathcal{T}}_{k}$ of size $n \ge 2$
into the root node and its subtrees. Let us assume that the out-degree of the root of $T$
is $r \ge 1$. After an order preserving relabelling the subtrees $T_{1}, \dots, T_{r}$
are itself increasingly $k$-tuple labelled weighted ordered trees
of certain sizes $s_{1}, \dots, s_{r}$. Since the root of $T$ is always labelled by
$(1,\dots,1)$ and the remaining labels are distributed over the
nodes of $T_{1}, \dots, T_{r}$ we obtain the following recurrence for the numbers
$T_{n}$, with $T_{1}=1$:
\begin{equation}\label{eqnc2}
  T_{n} = \sum_{r \ge 1} \varphi_r \sum_{s_{1}+\cdots+s_{r}=n-1}
  \binom{n-1}{s_{1}, s_{2}, \dots, s_{r}}^k   \cdot T_{s_{1}} \cdot T_{s_{2}} \cdots T_{s_{r}}, \quad
  n \ge 2.
\end{equation}
Note that the factor $\varphi_r$ is appearing, since we are considering weighted ordered trees.
Translating this recurrence into a differential equation for the generating function $T(z)$ is straightforward and yields the stated result.
\end{proof}

To get a connection to hook-length formul\ae{} we only have to take into account that the number $|\widetilde{\mathcal{L}}^{[k]}(T)|$ of different increasing $k$-tuple labellings of a given tree $T$ of size $n$ with distinguishable nodes is given by
\begin{equation*}
  |\widetilde{\mathcal{L}}^{[k]}(T)| = |\widetilde{\mathcal{L}}^{[1]}(T)|^{k} = \frac{(n!)^{k}}{\prod_{v \in T} \big(h_{v}^{k}\big)},
\end{equation*}
where $|\widetilde{\mathcal{L}}^{[1]}(T)| = |\mathcal{L}^{[1]}(T)|$ is well-known, see Lemma~\ref{HookBijOutlookLem1}.

\begin{prop}
\label{HookBijOutlookLem2}
Given a family $\widehat{\mathcal{T}}_{k}$ of increasingly $k$-tuple labelled weighted ordered trees with degree-weight generating function $\varphi(t)$.
Then, the family $\mathcal{O}$ of ordered trees satisfies the following hook-length formula:
\[
\sum_{T\in\mathcal{O}(n)}\prod_{v\in T} \left(\frac{\varphi_{\odeg(v)}}{h_v^k} \right) = \frac{T_n}{(n!)^k}.
\]
\end{prop}

\section*{Outlook}
We note that using the combinatorial setup presented in this work it is possible to analyze certain tree-shape parameters like the root degree; this will be discussed elsewhere.


\begin{thebibliography}{00}

\bibitem{AbramowitzStegun1972}
M.~Abramowitz and I.~A.~Stegun, \emph{Handbook of Mathematical Functions with Formulas, Graphs, and Mathematical Tables}, New York: Dover Publications, 1972.

\bibitem{BerFlaSal1992}
F.~Bergeron, P.~Flajolet and B.~Salvy, Varieties of increasing trees,
\emph{Lecture Notes in Computer Science} 581, 24--48, 1992.

\bibitem{Blasius1908}
H.~Blasius, Grenzschichten in Fl\"ussigkeiten mit kleiner Reibung, \emph{Z. Math. Phys.} 56, 1--37, 1908. 
English translation available at \url{http://naca.central.cranfield.ac.uk/reports/1950/naca-tm-1256.pdf}.

\bibitem{Boutt2003II}
J.~Bouttier, P.~Di Francesco and E.~Guitter, Random trees between two walls: Exact partition function,
\emph{Journal of Physics. A. Mathematical and Theoretical} 36, 12349--12366, 2003.

\bibitem{BlasiusBoyd1999}
J.~P.~Boyd, The Blasius function in the complex plane, \emph{Experimental Mathematics} 8 (4), 381--394, 1999.

\bibitem{Chen2009}
W.~Y.~C.~Chen, O.~X.~Q.~Gao, and P.~L.~Guo, Hook length formulas for trees by Han's expansion.
\emph{Electronic Journal of Combinatorics} 16, \#R62, 2009.

\bibitem{CheYan2008}
W.~Y.~C.~Chen and L.~L.~M.~Yang, On Postnikov's hook length formula for binary trees.
\emph{European Journal of Combinatorics} 29, 1563--1565, 2008.

\bibitem{Don1975}
R.~Donaghey, Alternating permutations and binary increasing trees,
\emph{Journal of Combinatorial Theory, Series A} 18, 141--148, 1975.

\bibitem{Drmota2013}
M.~Drmota, Embedded trees and the support of the ISE, \emph{European Journal of Combinatorics} 34, 123--137, 2013. 

\bibitem{Dumont1979}
D.~Dumont, A combinatorial interpretation for the Schett recurrence on the Jacobian elliptic functions, \emph{Mathematics of Computation} 33, 1293-1297, 1979.

\bibitem{Dumont1981}
D.~Dumont, Une approche combinatoire des fonctions elliptiques de Jacobi, \emph{Advances in Mathematics} 1, 1--39, 1981.

\bibitem{BlasiusFinch2008}
S.~Finch, Prandtl-Blasius Flow, manuscript, 2008. Online available at \url{www.people.fas.harvard.edu/~sfinch/csolve/bla.pdf}

\bibitem{Flajo1981}
P.~Flajolet, Combinatorial aspects of continued fractions, \emph{Discrete Mathematics} 32, 125--161, 1980.

\bibitem{FlaFran1989}
P.~Flajolet and J.~Fran\c{c}con, Elliptic functions, continued fractions and doubled permutations, \emph{European Journal of Combinatorics} 10, 235--241, 1989.

\bibitem{FlaGabPek2005}
P.~Flajolet, J.~Gabarr{\'{o}} and H.~Pekari, Analytic urns,
\emph{Annals of Probability} 33, 1200--1233, 2005.
1989.


\bibitem{FlaSed2009}
P.~Flajolet and R.~Sedgewick, \emph{Analytic combinatorics},
Cambridge University Press, Cambridge, 2009.

\bibitem{FoaHan2010}
D.~Foata and G.-N.~Han, The doubloon polynomial triangle.
\emph{The Ramanujan Journal} 23, 107--126, 2010.

\bibitem{Fra1976}
J.~Fran\c{c}on, Arbres binaires de recherche: Propri\'{e}t\'{e}s combinatoires et applications. \emph{RAIRO Informatique Th\'{e}orique et Applications} 10, 35--50, 1976.

\bibitem{GesSeo2006}
I.~Gessel and S.~Seo, A refinement of Cayley's formula for trees, \emph{Electronic Journal of Combinatorics} 11(2), \#R27, 2006.

\bibitem{BlasiusHager2003}
W.~H.~Hager, Blasius: A life in research and education, \emph{Experiments in Fluids} 34, 566--571, 2003.


\bibitem{Han2008}
G.-N.~Han, Discovering hook length formulas by an expansion technique,
\emph{Electronic Journal of Combinatorics} 15, \#R133, 2008.

\bibitem{JanKubPan2011}
S.~Janson, M.~Kuba and A.~Panholzer, Generalized Stirling permutations, families of increasing trees and urn models,
\emph{Journal of Combinatorial Theory, Series~A} 118, 94--114, 2011.

\bibitem{KoecherKrieg}
M.~Koecher and A.~Krieg, \emph{Elliptische Funktionen und Modulformen}, Springer-Verlag, Berlin, 2007.

\bibitem{KubPan2007}
M.~Kuba and A.~Panholzer, On the degree distribution of the nodes in increasing trees,
\emph{Journal of Combinatorial Theory, Series~A} 114, 597--618, 2007.

\bibitem{KubPan2012}
M.~Kuba and A.~Panholzer, Bilabelled increasing trees and hook-length formulas,
\emph{European Journal of Combinatorics} 33, 248--258, 2012.

\bibitem{KubPan2013}
M.~Kuba and A.~Panholzer, A unifying approach for proving hook-length formulas for weighted
tree families, \emph{Graphs and Combinatorics} 29, 1839--1865, 2013.

\bibitem{KuzPakPos1994}
A.~G.~Kuznetsov, I.~M.~Pak and A.~E.~Postnikov, Increasing trees and alternating permutations,
\emph{Russian Mathematical Surveys} 49(6), 79--114, 1994.

\bibitem{MahSmy1995}
H.~Mahmoud and R.~Smythe, A survey of recursive trees,
\emph{Theoretical Probability and Mathematical Statistics} 51, 1--37, 1995.

\bibitem{PanPro1998}
A.~Panholzer and H.~Prodinger, An analytic approach for the analysis of rotations in fringe-balanced binary search trees, \emph{Annals of Combinatorics} 2, 173--184, 1998.

\bibitem{Par1994}
S.~K.~Park, The $r$-multipermutations, \emph{Journal of Combinatorial Theory, Series A} 67, 44--71, 1994. 

\bibitem{Pos2009}
A.~Postnikov, Permutohedra, associahedra, and beyond.
\emph{International Mathematics Research Notices, IMRN}, no.~6, 1016--1106, 2009.

\bibitem{Pou1989}
C.~Poupard, Deux propri\'{e}t\'{e}s des arbres binaires ordonn\'{e}s stricts.
\emph{European Journal of Combinatorics} 10, 369--374, 1989.

\bibitem{ProUrb1983}
H.~Prodinger and F.~J.~Urbanek, On monotone functions of tree structures, \emph{Discrete Applied Mathematics} 5, 223--239, 1983.

\bibitem{Sloane}
N.~J.~A.~Sloane, \emph{The On-Line En\-cyclopedia of Integer Sequences (OEIS)}, Online availaible at\\ \texttt{www.research.att.com/~njas/sequences/}, 2009.

\bibitem{Sta1986}
R.~Stanley, \emph{Enumerative Combinatorics}, Vol.~I, Wadsworth \& Brooks/Cole, 1986.

\bibitem{Viennot1980}
G.~Viennot, Une interpr\'etation combinatoire des coefficients des d\'eveloppements en s\'erie enti\`ere des fonctions elliptiques de Jacobi, \emph{Journal of Combinatorial Theory, Series A} 29, 121--133, 1980.

\end{thebibliography}
\end{document}